\documentclass[reqno, 11pt, letterpaper]{amsart}
\usepackage{mathrsfs,graphicx,url, amssymb}
\usepackage[usenames,dvipsnames]{xcolor}
\usepackage[colorlinks=true,linkcolor=Black,citecolor=Black, urlcolor=Black]{hyperref}
\usepackage[foot]{amsaddr}

\DeclareMathOperator \re {Re}
\DeclareMathOperator \im {Im}
\def \pr0{\mathcal{P}_0}

\newtheorem{thm}{Theorem}
\newtheorem{lem}{Lemma}
\newtheorem{prop}[lem]{Proposition}
\newtheorem{cor}[lem]{Corollary}
\theoremstyle{definition}
\newtheorem{defin}{Definition}

\numberwithin{equation}{section}
\numberwithin{lem}{section}

\newcommand{\Pe}{\mathcal P_e}
\newcommand{\supp} {\operatorname{supp}}
\newcommand{\rank} {\operatorname{rank}}
\newcommand{\loc}{\operatorname{loc}}
\newcommand{\Real}{\mathbb{R}}
\newcommand{\Complex}{\mathbb{C}}
\newcommand{\Integers}{\mathbb{Z}}
\newcommand{\bp}{\mathbb{B}}
\newcommand{\Natural}{\mathbb{N}}

\newcommand{\bpr}{\mathbb{B}_{r_1}}

\newcommand{\mch}{\mathcal{H}}

\newcommand{\mcg}{\mathcal{G}}
\newcommand{\fc}{c}
\newcommand{\ve}{\varepsilon}

\usepackage[letterpaper,margin=1in]{geometry}

\title[Low energy resolvent expansions in dimension two]{Low energy resolvent expansions in dimension two}
\author[T. J. Christiansen and K. Datchev]{T. J. Christiansen and K. Datchev}
\address{Department of Mathematics, University of Missouri, Columbia, MO 65211 USA}
\email{christiansent@missouri.edu}
\address{Department of Mathematics, Purdue University, West Lafayette, IN 47907 USA}
\email{kdatchev@purdue.edu}
\begin{document}
\begin{abstract}
The behavior of the resolvent at low energies has  implications for many kinds of asymptotics, including for the scattering matrix and phase, for the Dirichlet-to-Neumann map, and for wave evolution. In this paper we present a robust method, based in part on resolvent identity arguments following Vodev and boundary pairing arguments following Melrose, for deriving such expansions, and implement it in detail  for 
compactly supported perturbations of the Laplacian on $\Real^2$.  We obtain precise results for general self-adjoint black box perturbations, in the sense of Sj\"ostrand--Zworski, and also for some non-self-adjoint ones.  The most important terms are the most singular ones, and we compute them in detail, relating them to spaces of zero eigenvalues and resonances.
\end{abstract}
\maketitle

\section{Introduction}

This paper proves resolvent asymptotics for two-dimensional Euclidean scattering in the low energy limit. We begin  with some   important consequences of our expansions. 

\subsection{Examples and applications}\label{s:exapp} The simplest concrete scatterers to which our general results apply are Schr\"odinger potentials and obstacles. In this section, we introduce these examples, and discuss wave decay rates, the role of logarithmic capacity, and perturbations of zero-energy resonances and eigenvalues.

\subsubsection{Wave expansions} \label{s:wave} Let $f \in C_c^\infty(\mathbb R^2)$, and let $w$ solve the wave equation
\[
( \partial_t^2 +P) w(x,t) = 0 ,\qquad  w(x,0) = 0, \ \partial_t w(x,0) = f(x),
\]
where, to begin with, we consider the following two kinds of operators $P$:
\begin{enumerate}
 \item The free Laplacian $P=-\Delta$. Then,  by explicit calculation, we have
\begin{equation}\label{e:freew}w(x,t) = \frac 1 {2\pi t} \int \Big(1 - \frac{|x-x'|^2}{t^2}\Big)^{-\frac 12} f(x')dx' = \frac 1 {2\pi t} \int f + O(t^{-3}),\end{equation} 
uniformly for large $t$, and for $x$ varying in a fixed compact set. 

\vspace{3mm}

\item  The Schr\"odinger operator  $P = -\Delta+V$, with $V \in L^\infty(\mathbb R^2)$  nontrivial, nonnegative and compactly supported. Then \eqref{e:freew} is replaced by
\begin{equation}\label{e:w2}
 w(x,t)=   \frac 1 {2\pi t (\log t)^2} \Big(\int U_{\log} f\Big)U_{\log}(x) + O(t^{-1}(\log t)^{-3}),
\end{equation}
where $U_{\log}$ obeys $PU_{\log}=0$ and $U_{\log}(x) \sim \log|x|$ as $|x| \to \infty$. 
\end{enumerate}
 
Observe that, in the free case, the only bounded solutions to $Pu=0$ are constants. Further, the leading term of \eqref{e:freew} decays like $t^{-1}$, with coefficient given by a sort of projection onto constants. 
By contrast, in the perturbed case, there are no bounded solutions to $Pu=0$, and the leading term of \eqref{e:w2} has the faster decay rate $t^{-1}(\log t)^{-2}$. It is given by a sort of projection onto $U_{\log}$. 

We say that the free case $P=-\Delta$ has a resonance at zero because of the existence of bounded solutions to $Pu=0$, while for nonnegative $V$ the perturbed case $P = -\Delta+ V$ has  no resonance at zero because of the nonexistence of such solutions; see Definition \ref{d:0res} below. This accounts for the difference between \eqref{e:freew} and \eqref{e:w2}.

The wave asymptotic \eqref{e:w2} follows from the corresponding  resolvent asymptotic \eqref{e:no0resser} below, via Stone's formula for the wave propagator in terms of the resolvent. See \cite{CDY2}, where more general and complete wave asymptotics are also discussed. More specifically, Stone's formula allows us to write $w(t,x)$ as an oscillatory integral of the resolvent $(P - \lambda^2)^{-1}$, and $t \to +\infty$ asymptotics of $w$ correspond to $\lambda \to 0$ asymptotics of the resolvent.

This correspondence holds for more general operators, including metric and obstacle scattering operators. A particularly nice case, thanks to recent semiclassical estimates of Vacossin \cite{vac,vac0}, is scattering by several convex obstacles, depicted in Figure \ref{f:bumpob}. Let $\mathcal O = \bigcup_{j=1}^J \mathcal O_j$, where the $\mathcal O_j$ are convex compact subsets of $\mathbb R^2$ with smooth boundary satisfying   the Ikawa no-eclipse condition: if $i$, $j$, and $k$ are all different, then $ \mathcal O_i$ does not intersect the convex hull of $ \mathcal O_j \cup \mathcal O_k$. Let $P$ be the Dirichlet Laplacian on $\Omega = \mathbb R^2 \setminus  \mathcal O$, and let $f \in C_c^\infty(\Omega)$. Then \eqref{e:w2} holds, where  $\Delta U_{\log}=0$ in $\Omega$, $U_{\log}|_{\partial \Omega} = 0$, and $U_{\log}(x) \sim \log|x|$ as $|x| \to \infty$.

\begin{figure}[ht]
\includegraphics[width=15cm]{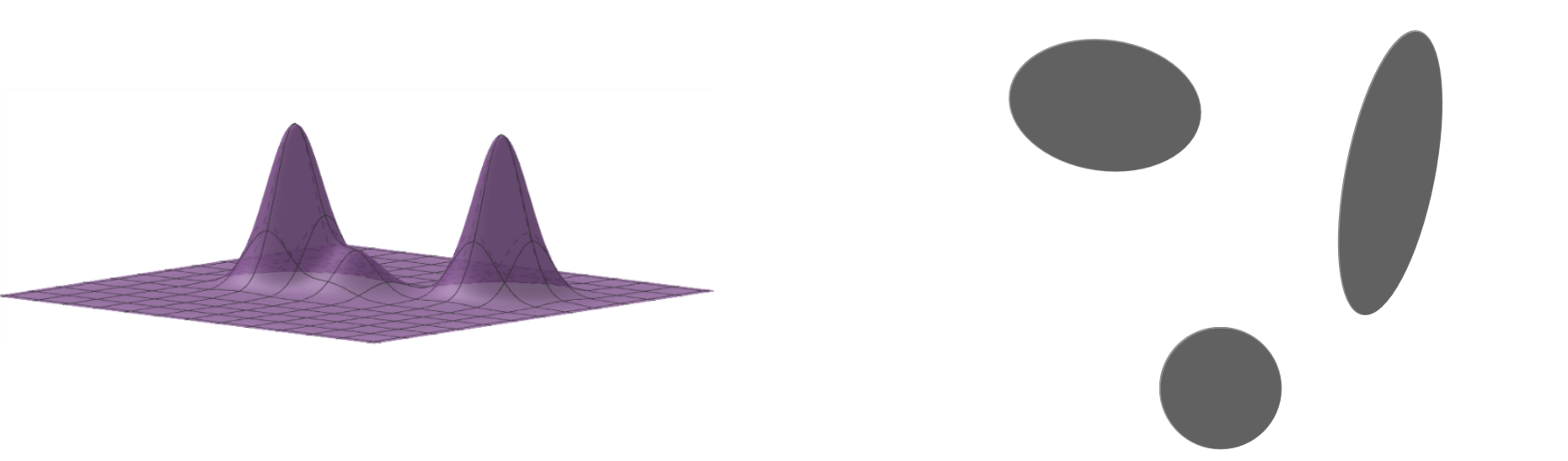}
\caption{Two examples of scatterers for which the wave asymptotics \eqref{e:w2} hold, a Schr\"odinger operator with nonnegative potential on the left, and three convex obstacles on the right.}\label{f:bumpob}
\end{figure}

\subsubsection{Logarithmic capacity and Dirichlet obstacle scattering}
A classical problem in scattering theory is {\em obstacle scattering.}  Let $\mathcal{O}\subset \Real^2$ be an {\em obstacle}, that is, a compact set such that $\Real^2\setminus \mathcal{O}$ is connected. This is a generalization of the setting of several convex obstacles discussed above. The Dirichlet problem on the exterior domain $\Real^2\setminus \mathcal{O}$ involves studying solutions of 
\begin{equation}\label{eq:Dp}
\left\{ \begin{array}{r l}
(-\Delta -\lambda^2)u=0 \;& \text{on }\;\Real^2\setminus \mathcal{O}\\
u =0 & \text{on }\; \partial \mathcal{O}.
\end{array}\right.
\end{equation}
If $\mathcal{O}$ is not polar, for example, if $\mathcal{O}$  contains a line segment, then there are no nontrivial bounded solutions of \eqref{eq:Dp} with $\lambda=0$.  There is, however, a solution $U_{\log}$ of 
$-\Delta U_{\log} =0$ on the exterior of~$\mathcal{O}$, satisfying the Dirichlet boundary condition, such that $U_{\log} (x)=\log|x|- C(\mathcal{O})+O(1/|x|)$ as $|x|\rightarrow \infty$.  The 
constant $C(\mathcal{O})$ is the logarithm of the logarithmic capacity of $\mathcal{O}$.

The function $U_{\log}$ appears in the asymptotic expansion of the cut-off resolvent of the Dirichlet Laplacian on the exterior of $\mathcal{O}$, as  does $C(\mathcal{O})$.
The constant $C(\mathcal{O})$ also appears in the expansion of  other fundamental quantities in scattering theory, including the Dirichlet-to-Neumann map, the scattering matrix, and the scattering phase; see \cite{ChDaob} for more on these expansions.  The scattering phase, $\sigma(\lambda)$, is, up to normalization, the logarithm of the determinant of the scattering matrix, and is important for its relation to the spectral shift function.
  If $\mathcal{O}$ is not polar, then
\begin{equation}\label{eq:sp}\sigma(\lambda) = \frac{1}{2\pi i} \log \left(1 + \frac{i\pi}{\log \lambda -\log 2+\gamma +C(\mathcal{O})-\pi i/2} \right) +O(\lambda^2 \log \lambda),\; \text{as}\;
\lambda \downarrow 0.
\end{equation}
The leading terms of this expansion were found by \cite{hassellzelditch} (to error $(\log \lambda)^{-2}$) and \cite{mcg} (to error $(\log \lambda)^{-4})$ using different techniques 
which are specific
to Dirichlet obstacle scattering.  
However, using our resolvent expansion and an identity of Petkov--Zworski \cite{pz} for the scattering matrix, one can show following the technique of \cite[Section 3]{ChDaob},
that for a large class of perturbations of $-\Delta$ on $\Real^2$ the scattering phase has behavior like \eqref{eq:sp} at $0$.

\subsubsection{Perturbations of  resonances or eigenvalues at zero}

The structure of 
the bottom of the continuous spectrum is particularly rich in two-dimensional scattering. In particular, there are two kinds of resonances at $0$, distinguished by the behavior of the corresponding
resonant states at infinity.  To illustrate this, let $V\in L^\infty_c(\Real^2;\Real)$ and consider the Schr\"{o}dinger operator $-\Delta +V$.  If there is a $u \in  L^\infty(\mathbb R^2) \setminus L^2(\Real^2)$ satisfying $(-\Delta +V)u=0$, then $-\Delta+V$ is said to have a resonance at $0$, and $u$ is a corresponding resonant state.  If $u(x)=O(|x|^{-1})$ as $x\rightarrow \infty$, then $u$ is called a $p$-resonant state; otherwise $u$ is called an $s$-resonant state.
The presence of each kind of resonant state leads to a different kind of singularity in the resolvent expansion at $0$.

Let $V_{\ve}=V_0+\ve V_1$, where $V_0,\; V_1\in L^\infty(\Real^2)$ are compactly supported, real-valued potentials, and $V_1\geq 0$ is nontrivial.  
It is classical \cite{ks} that if $-\Delta +V_0$ has a resonance or eigenvalue at $0$, then for $\ve<0$, $-\Delta +V_{\ve}$ has a negative eigenvalue $-\lambda_{\ve}^2$, with $-\lambda_{\ve}^2\uparrow 0$ as $\ve \uparrow 0$. But something interesting happens when $\ve$ increases from $0$: an $s$-resonance disappears, while a $p$-resonance or an eigenvalue persists as a nonzero resonance, with quite different  $\ve$-dependence in the perturbed $p$-resonance case compared to the perturbed eigenvalue case.  For small positive $\ve$, the effect of the resonance resulting from the perturbation can be easily observed by the
Breit-Wigner peak it causes in the scattering phase -- see Figure \ref{f:bw}.
The different kinds of behavior that result from perturbing an $s$-resonance, a $p$-resonance, and an eigenvalue at $0$ are explained by the different singular behavior of the resolvent at $0$ in these three cases.  See \cite{CDG} for more on this topic.

\begin{figure}[ht]
\includegraphics[width=15cm]{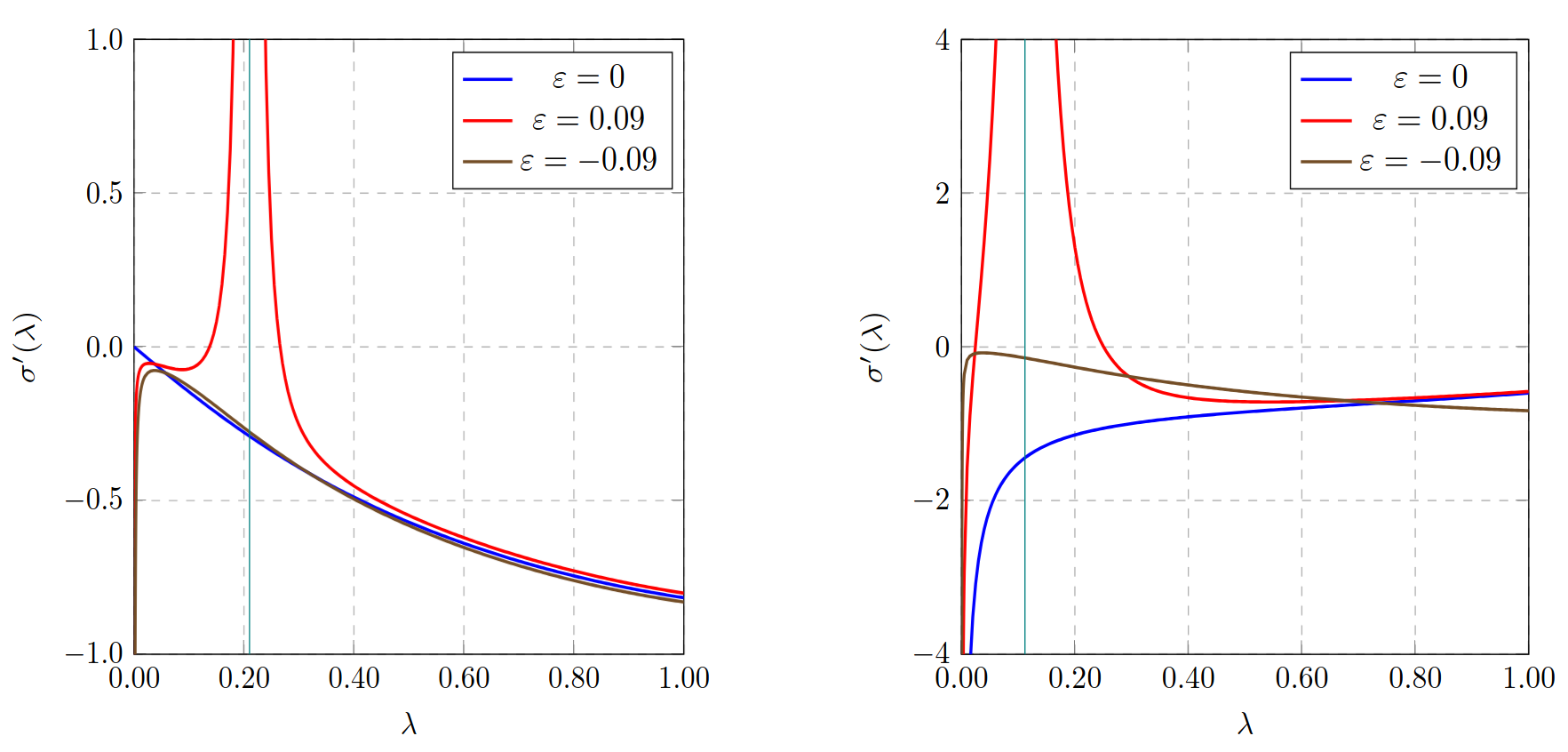}
\caption{Two examples of numerically striking Breit--Wigner peaks in the scattering phase. These peaks appear for $\varepsilon>0$ where a complex resonance occurs   near the real axis, corresponding to   the perturbation of a zero resonance/eigenvalue for $\varepsilon=0$. The peak on the left is more than twenty times taller than the already significant one on the right, because the left one corresponds to a zero eigenvalue while the right one  corresponds to a zero resonance. See \cite{CDG} for details.}\label{f:bw}
\end{figure}

\subsection{Our methods} 
The main results in this paper are low energy resolvent asymptotics for a broad class of operators generalizing the above examples, including black-box  operators in the sense of Sj\"ostrand and Zworski \cite{sz}, as well as certain non-self-adjoint  operators. 

Two main tools are an identity due to Vodev \cite{vo14}, and a boundary pairing functional.  (See Sections \ref{s:vod} and \ref{s:bp} respectively.)  Let $P$ denote one of the  operators introduced in Section \ref{s:exapp}.   Vodev's identity relates the $P$-resolvent $(P-\lambda^2)^{-1}$  to the free resolvent $(-\Delta - \lambda^2)^{-1}$.    Many such identities are used throughout spectral and scattering theory, but Vodev's identity has two features that make it particularly useful. The first is its generality, applying to any pair of operators which are compactly supported perturbations of one another. The second is that it involves two spectral parameters in an almost symmetric way; we denote these $\lambda$ and $z$, and in our arguments we take sometimes $\lambda \to 0$ and sometimes $z \to 0$.

Low energy asymptotics for the free resolvent follow from the standard formula for its Schwartz kernel in terms of a Hankel function. To compute asymptotics of $(P-\lambda^2)^{-1}$, we use Vodev's identity to reduce to the problem of inverting an operator $I+K(\lambda)$, where $K(\lambda)$ is compact.
In the case where the norm of the cut-off resolvent of $P$ does not grow too fast at the origin, we can show that  $I+K(\lambda)$ is 
invertible near $\lambda=0$ with a convenient expansion for its inverse.   On the other hand,  in the general case  a version of the 
analytic Fredholm theorem shows the existence of a less nice expansion of the inverse at the origin.  Further work,  using both Vodev's identity and 
a kind of boundary pairing identity similar to that 
of \cite{tapsit},  allows us to determine the coefficients of the first few singular terms in the resolvent expansion.    In particular, we show the connection between
the existence of an eigenvalue at $0$, a $p$-resonance, an $s$-resonance, and corresponding singularities of the resolvent.

Our methods work beyond the setting described in this paper. For example, in \cite{CDY1}, we apply them to Aharonov--Bohm operators. These have the form $P = (-i\nabla - \vec A)^2$ with  the support of $\vec A$ extending to infinity, and their important feature is that the magnetic potential has physical effects beyond the support of its magnetic field: see \cite{at,ds,my,cf,f24} for recent work on such operators.  Our methods also work for asymptotically conic manifolds: see \cite[Section~1.3]{CDY2}. In this paper, we focus on dimension two because of the particularly rich phenomena this setting affords, but any other dimension works just as well.

\subsection{Resolvent expansions}\label{s:mainres}

Our main resolvent expansions hold in the abstract framework of Section \ref{s:bb}, which includes the black box setting of Sj\"ostrand and Zworski \cite{sz}, as well as certain non-self-adjoint problems. 
The following examples are especially important:
\begin{enumerate}
\item Let $P= - \Delta +V$, where $V$ is a  bounded, compactly supported function.   For Theorem \ref{t:nre} we will need $V\in L^\infty_c(\Real^2;\Real)$. For Theorem \ref{t:resexp} it is enough if 
$V\in L^\infty_c(\Real^2;\Complex)$ and
$\re V\geq 0$.
\item Let $P$ be the Dirichlet or Neumann Laplacian $-\Delta$ on $\mathbb R^2 \setminus \mathcal O$, where $\mathcal O$ is a compact set with $C^\infty$ boundary and $\mathbb R^2 \setminus  \mathcal O$ is connected. 
\end{enumerate}

To describe the spaces of solutions to $Pu=0$ needed for our main statements, we introduce the following definitions. Let $\Omega = \mathbb R^2$ in the first case and $\Omega = \mathbb R^2 \setminus \mathcal O$ in the second case. Let $\mathcal H =  L^2(\Omega)$, let $\mathcal D$ be the domain of $P$, let $\mathcal H_{c}$ be the set of functions in $\mathcal H$ which have bounded support, and let $\mathcal D_{\text{loc}}$ be the set of functions which are locally in $\mathcal D$. 
  Recall that if a function $u$ is harmonic and polynomially bounded outside of a disk, then it grows or decays like a power of $|x|$ or like  $\log |x|$.
We accordingly define:
\begin{equation}\label{e:gldef}\begin{split}
\mcg_{l}&:=\{ u\in \mathcal{D}_{\loc} \colon  Pu=0,\; u(x)=O(|x|^{l}) \; \text{as $|x|\rightarrow \infty$}\},\\
\mcg_{\log}& := \{ u\in \mathcal{D}_{\loc} \colon  Pu=0,\; u(x)=O(\log |x|)\; \text{as $|x|\rightarrow \infty$}\}.
\end{split}\end{equation}
Note that $\mcg_{-2}$ is the zero eigenspace of $P$.

\begin{defin}\label{d:0res}
If $\mcg_{0} \ne \mcg_{-2}$, we say $P$ has a \textit{zero resonance}. If $u\in \mcg_0 \setminus \mcg_{-1}$, we call $u$ an $s$-resonant state and say $P$ has an $s$-resonance. If $u\in \mcg_{-1} \setminus \mcg_{-2}$, we  call $u$ a $p\,$-resonant state and 
 say $P$ has a $p\,$-resonance. 
 \end{defin} 

For example, the free Laplacian and  Neumann Laplacian always have a zero resonance with an $s$-resonant state (let $u$  be a constant) while the Dirichlet Laplacian never does. But 
the Schr\"odinger operator $-\Delta +V$ 
may have various combinations of $s$-resonance, $p\,$-resonance, and zero eigenvalue. 
See Section \ref{s:examples} for more on these examples.

Dimensions of quotient spaces of the $\mcg_l$ are especially relevant for the form of low-energy resolvent expansions. By standard  harmonic function expansions, recalled in \eqref{e:fcexp} below, the space of $s$-resonances (resp. $p\,$-resonances) has, modulo more rapidly decaying elements of the nullspace of $P$, dimension at most one (resp. two). In symbols,
\begin{equation}\label{e:dg0g-1}
 \dim  (\mcg_{0}/\mcg_{-1})  \leq 1 \quad \text{and} \quad \dim  (\mcg_{-1}/\mcg_{-2}) \leq 2.
\end{equation}

Recall that the resolvent $(P-\lambda^2)^{-1}$ continues meromorphically from the upper half plane to $\Lambda$, the Riemann surface of $\log \lambda$, as an operator $\mathcal H_{c}$ to $\mathcal D_{\text{loc}}$, i.e. from compactly supported functions in $\mathcal H$ to functions which are locally in $\mathcal D$.  We denote this meromorphic continuation by $R(\lambda)$. In other words, for any smooth and compactly supported $\chi$, $\chi R(\lambda) \chi$ is meromorphic on $\Lambda$, with $\|\chi R(\lambda) \chi\|_{\mch \to \mch}$ bounded away from poles.  See Section \ref{s:vod} below for a review of this.

Although a number of our results are valid for non-self-adjoint operators, for our first result we restrict ourselves to the self-adjoint case.   
\begin{thm} \label{t:nre}
Let $P$ be a self-adjoint operator chosen from among the above examples, or a general self-adjoint black-box perturbation of the Laplacian as defined in Section \ref{s:bb}.   Then for any
$\varphi_0>0$, we have the following expansion in the sense of
operators $\mch_c\rightarrow \mathcal{D}_{\loc}$:
\begin{equation}\label{e:resexp11}
 R(\lambda) =  \frac{-\Pe}{\lambda^2}   +   \frac{1}{\pi}\sum_{m=1}^{M}\frac{U_{w_m}\otimes U_{w_m} }{\lambda^2(\log \lambda -s_m)}  + B_{01}\log \lambda+ O(1),
\end{equation}
as $\lambda \to 0$ with $|\arg \lambda|<\varphi_0$. Here: 
\begin{itemize}
 \item The operator $\mathcal{P}_e$ is projection onto the zero eigenfunctions of $P$, i.e.  onto $\mcg_{-2}$.
\item The notation $U \otimes U$ means the operator mapping $f$ to $U\langle f , U\rangle_{\mathcal H}$, $M= \dim \mcg_{-1} / \mcg_{-2}$, the $U_{w_m}$ are certain elements of $\mcg_{-1}$, and the $s_m$ are certain constants: see Proposition \ref{p:B-2negk}.

\item  We have  $ B_{01}= \frac {-1}{2\pi}  U_0 \otimes U_0 -(\rho^2/2)\mathcal P_e {\bf 1}_{> \rho} \Pi_{-2} {\bf 1}_{> \rho} \mathcal P_e  $, where $U_0 \in   \mcg_0$,   $ {\bf 1}_{> \rho}$ is the characteristic function  of the set $\{x \colon |x| > \rho\}$,  $\Pi_{-2}$ is projection onto the $l= -2$ modes in the Fourier expansion \eqref{e:fcexp}, and  $\rho$ is large enough that $P = -\Delta$ when $|x|\ge\rho$: see Proposition~\ref{p:b01form}.
\end{itemize}
\end{thm}

\noindent\textbf{Remarks.}  1.  From the statement of Theorem \ref{t:nre}, it is clear that certain leading terms of \eqref{e:resexp11} vanish when certain types of zero resonance/eigenvalue are absent, i.e. when $\mcg_l = \{0\}$ for some $l \in \{-2,-1,0\}$. Specifically, if $\mcg_{-2}=\{0\}$, then $\|\chi R(\lambda) \chi\| = O(\lambda^{-2}(\log \lambda)^{-1})$; if $\mcg_{-1} = \{0\}$ then $\|\chi R(\lambda) \chi\|= O(\log \lambda)$; if $\mcg_{0} = \{0\}$, then $\|\chi R(\lambda) \chi\| = O(1)$.

\noindent 2. The expansion \eqref{e:resexp11} can be continued: see \eqref{eq:reb} for a general result,  Propositions~\ref{p:noneglog} and~\ref{p:easysum} for simplifications  when $\mcg_{-1}= \mcg_{-2}$, and Theorem \ref{t:resexp} for further simplifications when $\mcg_{-1}=\{0\}$.

Our results in Section \ref{s:fsrb} show that if $P$ is a (not necessarily self-adjoint) black-box perturbation of the Laplacian as defined in Section \ref{s:bb} and if $\mcg_{-1}=\{0\}$, i.e. if $P$ has neither eigenvalue $0$ nor a $p\,$-resonant state, 
then the resolvent of $P$ grows mildly in the sense of the bound $\|\chi R(\lambda) \chi\| = o(\lambda^{-2} (\log \lambda)^{-1})$: see \eqref{e:uhpbd} below.  In some cases, including Dirichlet obstacle scattering, it is straightforward to prove this bound directly: see Section \ref{s:examples}.  In any case, with the resolvent bound $\|\chi R(\lambda) \chi\| = o(\lambda^{-2} (\log \lambda)^{-1})$ we are able to prove the following more refined resolvent expansion.  The assumption of the bound on the cut-off resolvent in Theorem \ref{t:resexp}
allows an easier and more direct proof than that of Theorem \ref{t:nre}.

\begin{defin}\label{d:series} We say a series $\sum_{n} T_n(\lambda)$ of operators $T_n(\lambda) \colon \mch_c \to \mathcal D_{\text {loc}}$ \textit{converges absolutely, uniformly on sectors near zero} if, for any $\varphi_0>0$ and any
$\chi\in C_c^\infty(\Real^2)$  which is identically $1$ on the set where $P \ne -\Delta$, there is $\lambda_0>0$, such that the series $\sum_n \|\chi T_n(\lambda) \chi \|$ converges uniformly for all $\lambda \in \Lambda$ with $|\lambda| \le \lambda_0$ and $|\arg \lambda | \le \varphi_0$. 
 \end{defin}

\begin{thm}\label{t:resexp} Let $P$ be as in the above examples, or alternatively let $P$ be a more general black-box perturbation of the Laplacian as defined in Section \ref{s:bb}. Suppose the resolvent of $P$ grows mildly at $0$ in the sense of the bound $\|\chi R(\lambda) \chi\| = o(\lambda^{-2} (\log \lambda)^{-1})$. 

If $\mcg_0  \ne \{0\}$, i.e. if $P$ has a zero resonance, 
then there are operators $B_{2j,k} \colon \mathcal H_{\text{c}} \to \mathcal D_{\text{loc}}$ such that
\begin{equation}\label{e:0resser}\begin{split}
 R(\lambda) &= \sum_{j=0}^\infty \sum_{k=0}^{2j+1} B_{2j,k} \lambda^{2j}(\log \lambda)^k = B_{01} \log \lambda + B_{00} + B_{23} \lambda^2 (\log \lambda)^3 + \cdots.
\end{split}\end{equation}
If $k\not =0$ then $B_{2j,k}$ has finite rank. Moreover, if $P$ is self-adjoint, then 
$$B_{01}=-\frac{1}{2\pi} U_0\otimes U_0,$$
where $U_0 \in \mathcal G_0$ obeys $U_0(x)=1+O(|x|^{-1})$ as $|x|\rightarrow \infty$ and is specified in Lemma \ref{l:A01}.

If $\mcg_0 = \{0\}$, i.e. if $P$ has no zero resonance, 
then there are operators $B_{2j,k}, \ \widetilde B_{2j,-k} \colon \mathcal H_{\text{c}} \to \mathcal D_{\text{loc}}$ and a constant $a$ such that
\begin{equation}\label{e:no0resser}\begin{split}
R(\lambda)  &=  \sum_{j=0}^\infty  \left( \sum_{k=0}^{j}  B_{2j,k}  (\log \lambda)^k + \sum_{k=1}^{j+1}  \widetilde  B_{2j,-k}  (\log \lambda -a)^{-k} \right)\lambda^{2j}\\
&=  B_{00}  +   \widetilde B_{0,-1} (\log \lambda -a)^{-1} +  B_{21}   \lambda^2 \log \lambda + \cdots.
\end{split}\end{equation}
If $k\not =0$, then $B_{2j,k}$ and $\tilde B_{2j,-k}$ have finite rank. Moreover, if $P$ is self-adjoint, then 
\[
 \widetilde B_{0,-1} = \frac 1 {2\pi} U_{\log} \otimes U_{\log},  \ \  a = \log 2 - \gamma + \frac {\pi i }2 + \lim_{|x| \to \infty} \Big( U_{\log}(x)-\log|x|\Big), \ \ \gamma = -\Gamma'(1) = 0.577\dots,
\]
where $U_{\log}$ is the unique element of $\mcg_{\log}$  such that $U_{\log}(x) \sim \log|x|$ as $|x| \to \infty$.

The series \eqref{e:0resser} and \eqref{e:no0resser} converge absolutely, uniformly on sectors near zero.
\end{thm}
\noindent\textbf{Remark.}  By Corollary \ref{c:nspchar}, if $\mcg_{-1} = \{0\}$  then the hypothesis
$\|\chi R(\lambda) \chi\| = o(\lambda^{-2} (\log \lambda)^{-1})$ holds.

The expressions for $B_{0,1}$ in the first case and $\tilde B_{0,-1}$ in the second case are similar for non-self-adjoint $P$ and can be found in Lemmas \ref{l:A01} and \ref{l:B01}. Note that, by Theorem \ref{t:nre}, if $P=P^*$ then  \eqref{e:uhpbd} implies $\mcg_{-1} = \{0\}$ and hence $U_0$ is determined uniquely by the conditions $U_0 \in \mathcal G_0$ and $U_0(x)= 1+O(|x|^{-1})$ as $|x|\rightarrow \infty$.

\subsection{Background and context} \label{s:back}
Low frequency resolvent expansions have a long history in scattering theory, explicitly since the early results of MacCamy \cite{maccamy} and implicitly even before. 
 Because in dimension two
there are several types of resonance and eigenvalue at zero, each playing a different role, this dimension is more challenging than  any other -- compare the papers \cite{JeKa,je80,je84,BGD,jn} which study 
this problem for Schr\"{o}dinger operators with  real-valued potentials decaying sufficiently fast at infinity in dimensions respectively three, at least five, four, two  (with an additional restriction) and dimension
no greater than two.  

One of our contributions in Theorem \ref{t:nre} is to give a  statement with explicit leading terms for general black-box operators. Such a 
unified and explicit result, without need for separate cases for presence of different kinds of resonance and eigenvalue at $0$, 
 seems to be new even just for Schr\"odinger operators, and we handle them together with obstacle problems and many other examples.    In the absence of elements of the null space of $P$ which decay at infinity, Theorem \ref{t:resexp} refines these results, with a more direct proof, and generalizes them to certain non-self-adjoint operators.

By comparison, Vainberg \cite{vai89} has very general abstract results and many references, but the expansions there are not as explicit as ours. Explicit results corresponding to Theorem~\ref{t:resexp} were obtained using the framework of  \cite[Chapter X]{vai89} by Kleinman and Vainberg \cite{kv} in the setting of  exterior differential operators. More recently, another abstract framework was developed by M\"uller and Strohmaier in \cite{MuSt}, and our methods (in particular the proof of Proposition~\ref{p:rexpbasic}) have some commonality with it. It was applied to obtain explicit results, similar to but not as strong as our Theorem \ref{t:resexp}, for differential forms on manifolds by Strohmaier and Waters in \cite{sw}. Some analogous results have been obtained for magnetic Pauli and Dirac operators in \cite[Theorem~6.5]{kovarik}.

The structure of the resolvent expansion at zero has
consequences for the low-energy behavior of the spectral measure, scattering matrix, and scattering phase -- see e.g. \cite{JeKa,ChDaob, gmwz}.  Moreover, as discussed in Section \ref{s:wave}, it
impacts the 
  long-time asymptotics of solutions of the wave or time-dependent Schr\"odinger equation, see, e.g. \cite{JeKa, vai89, schlag, ergr, dz, hintz}. 
We explore some of these applications for the 
  scattering matrix and scattering phase of the Dirichlet 
  Laplacian in \cite{ChDaob}, for resonance and eigenvalue behavior near zero in \cite{CDG}, for Aharonov--Bohm operators in \cite{CDY1}, and  for wave evolution in \cite{CDY2}.
 
Wave decay results, such as  \eqref{e:freew} and \eqref{e:w2} and related results, have been much studied for decades. The field is too wide-ranging to survey here. Let us mention the seminal work of Morawetz \cite{m61}, and the surveys in  \cite[Epilogue]{lp89}, \cite[Chapter X]{vai89}, \cite{dr}, \cite{tat}, \cite{dz}, \cite{vasy}, \cite{sch21}, \cite{klainerman}. 



Another recent direction is discussed in the survey \cite{ADH} on subwavelength resonator systems. These show strong scattering of waves by small objects, by analyzing at low energies a non-self-adjoint problem which does not fit the assumptions of Section \ref{s:bb}, but is still amenable to our methods: see the Remark following Proposition \ref{p:rexpbasic} for a resolvent expansion in their setting. 

\subsection{Plan of the paper}

In Section \ref{s:prelim} we introduce our abstract framework, give examples, and prove some preliminary lemmas. In Section \ref{s:rs0} we prove Theorem \ref{t:resexp}, establishing our best and most direct resolvent expansions under the mild growth assumption \eqref{e:uhpbd}. The proof of Theorem \ref{t:resexp} has significant overlap with the proof of \cite[Theorem 1]{ChDaob}; moreover, the latter is a special case of the former.  
We accordingly refer the reader to corresponding parts of \cite{ChDaob} for certain details.

In Section \ref{s:fsrb} we prove our most general resolvent expansions. In Section \ref{s:renobd}, we refine the Section~\ref{s:fsrb} expansions in   the self-adjoint case: Theorem \ref{t:nre} follows from Propositions \ref{p:resexpsa},  \ref{p:b01form}, and~\ref{p:B-2negk}. Further information about negative powers of $\log \lambda$ is contained in Propositions~\ref{p:noneglog} and~\ref{p:easysum}. Sections \ref{s:fsrb} and \ref{s:renobd} do not use anything from Section \ref{s:rs0}.

\subsection{Notation and conventions}

\begin{itemize}
\item The spaces $\mathcal G_l$, $\mathcal G_{\log}$ of polynomially bounded solutions to $Pu=0$ are defined in \eqref{e:gldef}. 
\item   $\mathcal{P}_e$ is projection onto the zero eigenfunctions of $P$, i.e.  onto $\mcg_{-2}$.
\item $\Lambda$ is the Riemann surface of $\log \lambda$.
\item  $ {\bf 1}_{> \rho}$ is the characteristic function  of the set $\{x \colon |x| > \rho\}$.
\item The constant $\gamma_0 = \log 2 - \gamma + \frac {\pi i }2$ is defined in terms of Euler's constant $\gamma = -\Gamma'(1) = 0.577\dots$
\item The free resolvent notations $R_0(\lambda)$, $\tilde R_0(\lambda)$ and $R_{2j,k}$ are introduced in Section \ref{s:freeres}.
\item We use the complex inner product $a \cdot b = a^1\bar b^1 +a^2 \bar b^2$ for $a, b \in \mathbb C^2$.
\item The black-box notations $P$, $\mathcal H$, $\mathcal D$, $\mathcal B$, including the tensor product $\otimes$ and involution $u \mapsto\bar u$, are introduced in Section \ref{s:bb}.
\item The Fourier coefficients  $v_l, \; \fc_0,\; \fc_{\log}$ are introduced in \eqref{e:fcexp}.
\item The cutoff $\chi_1 \in C_c^\infty(\mathbb R^2)$, which  is $1$ near $\mathcal B$ and depends only on $|x|$, is introduced in the beginning of Section \ref{s:bp}. 
\item The radius $r_1>0$ is always large enough that $\chi_1(x) = 0$ when $|x|>r_1-1$, and sometimes taken larger so as to satisfy additional requirements.
\item The boundary pairing $\mathbb B$ is defined in \eqref{eq:bdrypair}.
\item The operators $K_1$, $K(\lambda)$, $F(\lambda)$ used in Vodev's identity are defined in the equations from \eqref{e:k1def} to \eqref{e:fdef}.
\item The operators $\tilde F(\lambda)$, $F_{2j,k}$, $A(\lambda)$, $D(\lambda)$, $D_{2j,k}$, the function $w$, and the complex numbers $\alpha(\lambda)$, $\alpha_{2j,k}$ are defined in \eqref{e:fseries} and Lemma \ref{l:rfk}.
\item The notation $U_w$,  where $w\in \Complex^2$, denotes an element of $\mcg_{-1}$ obeying $U_w(x)= w\cdot x/|x|^2+O(|x|^{-2})$.
Note that such a $U_w$ does not necessarily exist for arbitrary $w \in \mathbb C^2$, and when it does exist it is not necessarily unique.
\end{itemize}

\section{Preliminaries}\label{s:prelim}

\subsection{The free resolvent}\label{s:freeres}

Let $-\Delta$ be the nonnegative Laplacian on $\mathbb R^2$ and  $R_0(\lambda) = (-\Delta - \lambda^2)^{-1}$ its resolvent for $\lambda$ in the upper half plane. We briefly review some standard facts about $R_0$; see Section~2A of \cite{ChDaob} for references. The integral kernel of $R_0(\lambda)$ is given 
by 
\begin{equation}\label{e:h10ser}
 R_0(\lambda)(x,y)=  \frac{i}{4} H^{(1)}_0(\lambda|x-y|), \quad
 H^{(1)}_0(s) = \frac{2i}\pi \sum_{m=0}^\infty  \left(\log  s - \gamma_m\right)  \frac{(-s^2/4)^m}{(m!)^2},
\end{equation}
where 
\[
 \gamma_0=\log 2 -\gamma +\tfrac{\pi i}{2},  \quad \gamma = -\Gamma'(1) = 0.577\dots,\quad \gamma_m = \gamma_{m-1} + \tfrac 1m \text{ for } m \ge 1.
\]
For any $f \in L^2_{c}(\mathbb R^2)$ and $\lambda$ in the upper half plane,
\begin{equation}\label{e:r0exp}
 R_0(\lambda) f(x) = O(e^{-|x|\im \lambda}), \qquad \text{as } |x| \to \infty.
\end{equation}

It follows from  \eqref{e:h10ser} that   $R_0(\lambda)\colon L^2_{c}(\mathbb R^2)\to H^2_{\text{loc}}(\mathbb R^2)$ continues holomorphically from the upper half plane to $\Lambda$, the Riemann surface of $\log \lambda$. For each $\lambda \in \Lambda$  we write
\begin{equation}\label{e:r0series}
 R_0(\lambda) = \sum_{j=0}^\infty \sum_{k=0}^1  R_{2j,k}  \lambda^{2j}(\log \lambda)^k =  R_{01}\log \lambda + \tilde R_0(\lambda),
\end{equation}
where $\tilde R_0(\lambda)$ is defined by the equation, the $R_{2j,k}$ are  operators $L^2_{c}(\mathbb R^2)\to H^2_{\text{loc}}(\mathbb R^2)$, and the series converges absolutely, uniformly on sectors near zero. More explicitly,  the integral kernels of the leading terms are as follows, with asymptotics valid  for $y$ in a fixed compact set
and $|x|\rightarrow \infty$:
  \begin{equation} \begin{split}\label{eq:R2jk}
R_{01}(x,y)&=-\frac{1}{2\pi}  \\
R_{00}(x,y)&=-\frac{1}{2\pi}(\log |x-y| - \gamma_0)  
 \\ & = - \frac 1 {2\pi} (\log|x| -\gamma_0) +  \frac 1 {4\pi}\sum_{m=1}^\infty \frac 1 {m|x|^{2m}}   (2x\cdot y - |y|^2)^m   \\
R_{21}(x,y)& = \frac{1}{8\pi}|x-y|^2=\frac{1}{8\pi}(|x|^2-2x\cdot y+|y|^2)  \\
R_{20}(x,y)& = \frac{1}{8\pi}\big(\log|x-y|- \gamma_0 - 1\big) |x-y|^2  \\ & = 
\frac{1}{8\pi}\Big(\log|x|- \gamma_0-1-\frac{x\cdot y}{ |x|^2}\Big) |x-y|^2+O(|x|^{-2})|x-y|^2  \\
R_{41}(x,y)&
= \frac {-1}{128\pi} ((|x|^2 + |y|^2)^2 - 4 |x|^2 (x\cdot y) + 4 (x \cdot y)^2  - 4 (x\cdot y)|y|^2).
\end{split}\end{equation}


\subsection{Black-box setup and notation}\label{s:bb}
The operator $P$ will either be a black-box perturbation of $-\Delta$ on $\mathbb R^2$ in the sense of \cite{sz}; see also \cite[Chapter 4]{dz}, or a 
variant of this that allows for certain non-self-adjoint operators as described below.  
 We briefly review the basic definitions and results that we will need. Let $\mathcal H$ be a complex Hilbert space with orthogonal decomposition
\[
 \mathcal H = \mathcal H_0 \oplus L^2(\mathbb R^2 \setminus \mathcal{B}),
\]
where $\mathcal H_0$ is a separable Hilbert space and $\mathcal{B}$ is a fixed ball in $\mathbb R^2$. Define similarly
\[
  \mathcal H_{\text{loc}} = \mathcal H_0 \oplus L^2_{\text{loc}} (\mathbb R^2 \setminus \mathcal{B}), \qquad \text{and} \qquad    \mathcal H_{c} = \mathcal H_0 \oplus L^2_{c} (\mathbb R^2 \setminus \mathcal{B}).
\]
Let $u \mapsto u|_{\mathcal B}$ and $u \mapsto u|_{\mathbb R^2 \setminus \mathcal B}$ denote the orthogonal projections $\mathcal H \to \mathcal H_0$ and  $\mathcal H \to L^2(\mathbb R^2 \setminus \mathcal{B})$, and denote their natural extensions $\mathcal H_{\text{loc}} \to \mathcal H_0$ and  $\mathcal H_{\text{loc}} \to L^2_{\text{loc}}(\mathbb R^2 \setminus \mathcal{B})$ in the same way. For $\chi$ a function on $\mathbb R^2$ which is equal to a constant $c$ near $\mathcal{B}$, we define
\[
 \chi u = c u|_{\mathcal B} + \chi u|_{\mathbb R^2 \setminus \mathcal B}.
\]

We assume that at least one of (A1) or (A2) holds.

A1. For the ``classic" black-box operator, 
let $P$ be a self-adjoint operator on $\mathcal H$ with dense domain $\mathcal D$. We assume that $P$ is lower semi-bounded, that $u \mapsto ((P+z_0)^{-1}u)|_{\mathcal B}$ is compact for some $z_0$ with $\im z_0>0.$ 

A2. Alternatively, in order to allow certain non-self-adjoint operators $P$, we assume (as in equations (4.4.10) and (4.4.11) of \cite{dz}) that there is an involution defined on $\mathcal H$,
 $u\mapsto \bar u$, so that 
if $u\in L^2(\Real^2 \setminus \mathcal{B})$ is $0$ near $\mathcal{B}$, then $\bar{u}$ is the complex conjugate of $u$ as usual.  Moreover, we assume that if
 $c\in \Complex$ and $u\in \mathcal{H}$, then $\overline{cu} = \bar{c} \bar{u}$, and $\langle \bar u, \bar v\rangle_\mch = \langle v, u \rangle_\mch$.  An example to 
keep in mind is of course $\mch =L^2(\Real)$ with the involution  given by complex conjugation.   We need also some hypotheses on the operator $P$: it has dense domain $\mathcal{D}\subset \mathcal{H}$, if $u\in 
\mathcal {D}$ then $\bar u$ is in the domain of $P^*$, with 
$\overline{P^* \bar u}=Pu$.  Moreover, we assume  there is an $M>0$ so that  for all $u\in \mathcal{D}$, $\re \langle P u,u\rangle >-M\|u\|^2_\mch$, and $u \mapsto ((P+z_0)^{-1}u)|_{\mathcal B}$ is compact for some $z_0$ with $\im z_0>0.$

We also assume  $\mathcal D|_{\mathbb R^2 \setminus \mathcal B} \subset H^2(\mathbb R^2 \setminus \mathcal{B})$, that $(Pu)|_{\mathbb R^2 \setminus \mathcal B} = (-\Delta u)|_{\mathbb R^2 \setminus \mathcal B}$, and that if $f \in H^2(\mathbb R^2 \setminus \mathcal{B})$ and $f = 0$ near $\mathcal{B}$ then $f \in \mathcal D$. Put $\mathcal D_{\text{loc}} = \mathcal D|_{\mathcal B} \oplus H^2_{\text{loc}}(\mathbb R^2)$.   
We use a tensor product notation analogous to \cite[$(2.2.19)$]{dz}:
\[
 (g \otimes h)f = g \langle f, h \rangle_{\mathcal H},
\]
and we use the same inner product notation for pairing vectors in $\mathcal H_{\text{loc}}$ with vectors in $\mathcal H_{\text{c}}$.
By statements like
\[
u(x) = O(|x|^l), \qquad \text{or} \qquad u = O(|x|^l),
\]
as $|x| \to \infty$, as in \eqref{e:gldef}, we mean $u|_{\mathbb R^2 \setminus \mathcal B}(x) = O(|x|^l)$.  If $u\in \mathcal{D}_{\loc}$, by $Pu$ we mean 
$Pu=-\Delta (1-\chi)u+P\chi u\in \mathcal{H}_{\loc}$, where $\chi\in C_c^\infty(\Real^2)$ is $1$ in a neighborhood of $\mathcal{B}$.

We can now make sense of the spaces $\mcg_l$ and $\mcg_{\log}$ of functions in $\mathcal D_{\text{loc}}$ annihilated by $P$, defined in \eqref{e:gldef}. Recall that such functions have a large $|x|$ expansion in polar coordinates: if $\phi \in \bigcup_{l}\mcg_l$, then  there are coefficients $\fc_0(\phi)$, $\fc_{\log}(\phi)\in \Complex$, and $v_l(\phi)\in \Complex^2$, such that for $r$ large enough we have
\begin{equation}\label{e:fcexp}
\phi(r \cos \theta , r \sin \theta)=\fc_0(\phi)+\fc_{\log}(\phi)\log r + \sum_{l \in \mathbb  Z \setminus\{0\}} v_l(\phi)  \cdot (\cos(|l|\theta),\sin(|l|\theta)) r^{l}.
\end{equation}

Under the above assumptions, $R(\lambda)=(P-\lambda^2)^{-1}\colon \mathcal H \to \mathcal D$ is meromorphic for $\lambda$ in the upper half plane, and $R(\lambda)$ continues meromorphically as an operator from $\mathcal H_{\text{comp}}$ to $\mathcal D_{\text{loc}}$ to $\lambda \in \Lambda$; we  review these facts in Section \ref{s:vod} below.

We record the following straightforward lemma for future reference.
\begin{lem}\label{l:resA2}
Suppose the black-box operator $P$ and the Hilbert space $\mch$ satisfy (A2).  Let $f\in \mch$ and take $\im \lambda>0$, away from any square roots of eigenvalues of $P$. Then  $(R(\lambda))^*f= \overline{R(\lambda)\bar f}$.
\end{lem}
\begin{proof}
Note that our assumptions (A2) imply that 
$$ (P^*-\overline{\lambda}^2) \overline{R(\lambda) f}=\overline{(P-\lambda^2)R(\lambda)f}= \bar f$$
so that $(P^*-\overline{\lambda}^2)\overline{R(\lambda)\bar f}=f$.  Comparing this to $(P^*-\overline{\lambda^2})(R(\lambda))^*=I$ gives $(R(\lambda))^*f= \overline{R(\lambda)\bar f}$.
\end{proof}

For our results in Section \ref{s:rs0} we will  need to assume that the resolvent grows mildly at $0$ in the sense of the bound
\begin{equation}\label{e:uhpbd}
 \|\chi R(i\kappa) \chi \|_{\mathcal H \to \mathcal D} =o(\kappa^{-2} (\log \kappa)^{-1}), \  \text{  as } \kappa \to 0 \text{ along the positive real axis,}
\end{equation}
for any $\chi \in C_c^\infty(\mathbb R^2)$ which is 1 near $\mathcal B$.  Corollary \ref{c:nspchar} shows that if $P$ has neither eigenvalue $0$ nor a $p\,$-resonance state (i.e. if $\mcg_{-1} = \{0\}$), then the 
estimate \eqref{e:uhpbd} holds.  Below we consider some examples in which one can rather easily check that there are no nontrivial elements of the 
null space of  $P$  which decay at infinity.


\subsection{Examples}\label{s:examples} We now discuss further the basic examples we introduced earlier.

\subsubsection{Schr\"{o}dinger operators}\label{s:schr} Let  $P = -\Delta  + V$, with $V\in L^\infty_c(\Real^2;\Complex)$.     If either $\re V\geq 0$   or $\pm \im V \ge 0$ with $\im V \not \equiv 0$, then we can quickly show that $\mcg_{-1}=\{0\}$,
and thus  by
Corollary~\ref{c:nspchar}  the resolvent satisfies the estimate (\ref{e:uhpbd}).  Indeed, if $u \in \mcg_{-1}$, then 
$$0=\int_{|x|<\rho} Pu \overline{u} dx= \int_{|x|<\rho} (|\nabla u|^2 +V |u|^2)dx-\int_{|x|=\rho} \partial_r u \overline{u} dS_x.$$
By \eqref{e:fcexp}, $\lim_{\rho \rightarrow \infty }\int_{|x|=\rho} \partial_r u \overline{u} dS_x=0$.
Hence 
\begin{equation}\label{e:nablauvu} 
\int_{\Real^2} (|\nabla u|^2 +V |u|^2)dx = 0. 
\end{equation}
If $\re V\geq 0$, then taking the real part of \eqref{e:nablauvu} shows that $\nabla u\equiv 0$. Since $u(x)=O(|x|^{-1})$, we get $u\equiv 0$.   If $\pm \im V\geq 0$, then taking the imaginary part of \eqref{e:nablauvu} shows that $\im V |u|^2 \equiv 0$. Using  $\im V \not \equiv 0$ and  unique continuation \cite[Corollary 2]{robbiano}, we again get  $u\equiv 0$.


If  $\re V\geq 0$, one can prove the resolvent estimate \eqref{e:uhpbd} directly, without resorting to Corollary~\ref{c:nspchar}. One way is to argue as in Lemma 2.1 of \cite{ChDaob}: if $u = R(i\kappa)\chi f$, then
\begin{equation}\label{e:ladyzhenskaya}
 \|\nabla u\|^2_{L^2} + \kappa^2 \|u\|^2_{L^2} \le \langle \chi f,  u \rangle_{L^2} \le 
 \|\chi\|_{L^4} \|f\|_{L^2} \|\nabla u\|_{L^2}^{1/2} \| u\|_{L^2}^{1/2},
\end{equation}
which implies the stronger estimate $\|R(i\kappa) \chi \|_{\mathcal H \to \mathcal D} =O (\kappa^{-3/2})$.

We remark that the  paper \cite{CDG} studies low-energy behavior of resonances and the scattering phase for Schr\"{o}dinger operators, with explicit calculations for 
 Schr\"{o}dinger operators with circular well potentials.

\subsubsection{Obstacles}\label{s:obst} Let $P=-\Delta$ be the Dirichlet or Neumann Laplacian $-\Delta$ on $\Omega = \mathbb R^2 \setminus \mathcal O$, where $\mathcal O$ is a compact set with $C^\infty$ boundary and $\Omega$ is connected. 

To determine whether $P$ has a zero resonance or eigenvalue, we compute as in the proof of Theorem 4.19 of \cite{dz}. Let $u \in \mcg_0$, and take $\rho$ large enough that $|x| < \rho$ for all $x \in \mathcal O$. As in Section \ref{s:schr}, we have
\[
 0 \le \int_{\{x \in \Omega \colon |x| \le \rho\}} |\nabla u|^2 = - \int_{|x|=\rho} (\partial_r u) \bar u  = O(\rho^{-1}), \qquad \text{as }\rho \to \infty,
\]
and hence $u$ is constant. Thus $u$ can be nontrivial if and only if the boundary condition is Neumann. We conclude that $P$ 
has an $s$-resonance in the Neumann case, but not in the Dirichlet case. In either case there is no $p\,$-resonance and no zero eigenvalue, so the estimate (\ref{e:uhpbd}) holds by Corollary~\ref{c:nspchar}.

In the Dirichlet case  (\ref{e:uhpbd}) also follows directly from \eqref{e:ladyzhenskaya}, and it holds moreover for any obstacle $\mathcal O$ which is not polar: see \cite[Section 2.2]{ChDaob} for more details.

\subsection{Boundary pairing}\label{s:bp}

It is convenient to introduce a boundary pairing, similar to that from \cite[Section 6.1]{tapsit}, whose notation we now adapt to our setting. Let $\chi_1 \in C_c^\infty(\mathbb R^2;[0,1])$ be such that $\chi_1$ is $1$ near $\mathcal{B}$ and $\chi_1$ depends only on $|x|$.
For $r_1>0$ such that $\chi_1(x)=0$ for $|x|>r_1-1$, and $u, v\in H^2(\{ x\in \Real^2: \left| |x|-r_1\right|<1\})$, define
\begin{equation}\label{eq:bdrypair}
\bp(u,v) = \bp_{r_1}(u,v)=\int_{|x|=r_1} u\partial_r\overline{v}-(\partial_r u)\overline{v}.
\end{equation}
The circle $|x|=r_1$ plays the role of the boundary of the interaction region.\footnote{In \cite[Section 6.1]{tapsit} the `boundary' in `boundary pairing' refers to a boundary at infinity. See also \cite[Section 4.4.3]{dz} for another related usage.}

Define for $u_1,\; u_2 \in \mch_{\loc}$,
\begin{equation}\label{eq:ipr1}
\langle u_1,u_2 \rangle_{|x|<r_1}= \langle \chi_1^{1/2} u_1, \chi_1 ^{1/2}u_2   \rangle + \int_{|x|<r_1} (1-\chi_1) u_1 \overline{ u_2} dx,
\end{equation}
and define similarly $\langle u_1,u_2 \rangle_{|x|>r_1},\; \langle u_1,u_2 \rangle_{r_1<|x|<r_2}$. Note that this  definition is independent of the choice of $\chi_1$. 

\begin{lem} \label{l:bpid} Let $u_1\in \mathcal{D}_{\loc}$, $u_2\in \mch_{\loc}$ and suppose for each $\chi \in C_c^\infty(\Real^2)$ which is constant near $\mathcal{B}$, $u_2 \chi $ is in the domain of $P^*$.  Then
$$\bp(u_1,u_2)= \langle P u _1,u _2\rangle_{|x|<r_1}-\langle u_1, P^*u_2 \rangle_{|x|<r_1}.$$
\end{lem}
\begin{proof}
This follows in a straightforward way from Green's formula.
\end{proof}

The next two results are closely related.

\begin{lem}\label{l:sGreens} Let $\chi_1 \in C_c^\infty(\Real^2)$ be one in a neighborhood of $\mathcal{B}$, and such that $\chi_1(x)=0$ for $|x|>r_1-1$.
Let $u\in \mathcal{D}_{\loc}$, $v\in H^2_{\loc}(\Real^2)$.  Then
 $$\langle [-\Delta, \chi_1]u,v\rangle =\langle (1-\chi_1)u,\Delta v \rangle_{|x|<r_1} +\langle (1-\chi_1)Pu ,v \rangle_{|x|<r_1}
-\bp(u, v).$$
\end{lem}
\begin{proof}
We  write
\begin{align*} \langle [&-\Delta, \chi_1]u,v \rangle  = \langle (\Delta(1-\chi_1)+(1-\chi_1)P)u,v\rangle_{|x|<r_1}\\&
= \langle (1-\chi_1)u,\Delta v \rangle_{|x|<r_1} +\langle (1-\chi_1)Pu, v \rangle_{|x|<r_1}
+ \int_{|x|=r_1}(\partial_r u \overline{v}- u \partial_r\overline{v}),
\end{align*}
by using Green's formula.
\end{proof}

\begin{lem}\label{l:bpl}
Fix $\phi\in \mathcal{H}_c $ and $y_0\in \Real^2$ such that $|y_0|<r_1$. Then
$$\langle \phi, \overline{R_0(\lambda)}( \bullet, y_0)\rangle _{|x|>r_1}=- \bp_{r_1}( R(\lambda)\phi, \overline{R_0(\lambda)}(\bullet, y_0)).$$
\end{lem}
\begin{proof} By analytic continuation, it is enough to prove the equality for $\lambda =i|\lambda|$, $|\lambda|>0$, and $\lambda$ away from any poles of $R(\lambda)$.  
Recall that, by \eqref{e:r0exp}, 
$R_0(\lambda)(x, y_0)$ is exponentially decaying in $|x|$.  
Then, with $\psi_1 = R(\lambda) \phi$ and $\psi_2 = \overline{R_0(\lambda)}( \bullet, y_0)$, we have
\[
\langle \phi,  \psi_2\rangle _{|x|>r_1} = \lim_{\rho \rightarrow \infty}  \langle (P-\lambda^2)\psi_1, \psi_2\rangle _{\rho > |x|>r_1}
 =   - \bp_{r_1}( \psi_1, \psi_2),
\]
where we used  Green's formula in the annulus $\rho > |x|>r_1$ as in Lemma \ref{l:bpid}, the compact support of $\phi$, and $(-\Delta -\lambda^2)R_0(\lambda)(x, y_0)=0$ for $|x|>r_1$.
\end{proof}


We shall frequently take boundary pairings of functions which are harmonic near infinity, and so we record the following lemma, whose proof is a straightforward calculation which we omit.

\begin{lem}\label{l:bdrypair}
Let $\phi,\; \psi$ be harmonic and polynomially bounded for $x$ such that  $|x| > r_1-1$.  Then, in terms of the Fourier series expansion, \eqref{e:fcexp},
we have
\[
\frac 1 {2\pi} \bp(\phi,\psi)= \fc_0(\phi)\overline{\fc}_{\log}(\psi)-\fc_{\log}(\phi)\overline{\fc}_0(\psi)+\sum_{l\in \Integers\setminus\{0\}} l  v_{-l} (\phi) \cdot v_{l}(\psi).
\] 
\end{lem}

We will often use the following special case:
\begin{lem}\label{l:Ucommint}
If $\phi\in \mcg_{\log}$, then $ c_{\log}(\phi) = \frac {-1} {2\pi}\langle [\Delta,\chi_1] \phi, 1 \rangle = R_{01}[\Delta,\chi_1] \phi$.  
\end{lem}
\begin{proof} Use Lemma \ref{l:sGreens} with  $v=1$, Lemma \ref{l:bdrypair}, and the resolvent kernel formula \eqref{eq:R2jk}.
\end{proof}

We supplement the $\mathcal G_l$ spaces   with the following related spaces. For $l\in \Integers\setminus\{0\}$, let 
$$\mathcal{F}_l=\{ u \in C^\infty(\mathbb R^2\setminus \{0\}) \colon u(r\cos \theta, r \sin \theta)=(c_+e^{il\theta}+c_-e^{-il\theta})r^l,\; \text{  $c_{\pm}\in \Complex$}\},$$
and  $$\mathcal F_0= \{ u\in C^\infty(\mathbb R^2\setminus \{0\})\colon u(r\cos \theta, r \sin \theta)=c_0+c_1\log r, \ c_0 \text{ and } c_1\in \Complex \}.$$
Note from Lemma \ref{l:bdrypair} that $\bp$ is a nondegenerate pairing from $ \mathcal F_l \times \mathcal F_{-l}$ to $\mathbb C$, and that $\bp(u,v)=0$ if $u \in \mathcal F_l$, $v \in \mathcal F_k$, with $l\not = -k$.
Let 
\begin{equation}\label{eq:F'}
\mathcal F_l'= \{ u\in \mathcal F_l \colon \text{there exists } u'\in \mathcal{D}_{\loc} \text{ with } Pu'=0\;\text{and } u \sim u' \text{ as } r \to \infty\}.
\end{equation}

\begin{lem}\label{l:F'pairing}
Suppose $u\in \mathcal F_l'$ and $v\in \mathcal F_{-l}'$.  If $P=P^*$,  then $\bp(u,v)=0$.  If instead $P$ and $\mathcal{H}$ satisfy the hypotheses (A2), then
$\bp(u,\bar v)=0$.
\end{lem}
\begin{proof}
First we note that if $u'$ is as in (\ref{eq:F'}) and $v'$ is the analog for $v$, then
$\bp(u,v)=\bp(u',v')$ and $\bp(u,v)=\bp (u',\overline{v'}). $  Then the lemma follows from an application of Lemma \ref{l:bpid}.
\end{proof}
From Lemma \ref{l:F'pairing} and the nondegeneracy of the mapping $\bp: \mathcal F_l \times \mathcal F_{-l}\rightarrow \Complex$, we deduce
\begin{cor}\label{c:ubF's}
If $l\in \Natural_0$, then $\dim \mathcal F_l' + \dim \mathcal F_{-l}' \leq 2$.  
\end{cor}
Since $ \dim  (\mcg_{-1}/\mcg_{-2}) = \dim \mathcal F_{-1}'$ and $ \dim  (\mcg_{0}/\mcg_{-1}) \le \dim \mathcal F_0'$, Corollary \ref{c:ubF's} generalizes and sharpens \eqref{e:dg0g-1}. In particular, we obtain

\begin{cor}\label{c:g0glog}
 At most one of $\mcg_{\log} \setminus \mcg_0$ and $\mcg_0 \setminus \mcg_{-1}$ is nonempty.
\end{cor}

\subsection{Vodev's identity}\label{s:vod}

Let $z$ and $\lambda$ be in the upper half plane, away from any square roots of eigenvalues of $P$. To relate the resolvents of $P$ and $-\Delta$, we start by using 
$$ R(\lambda)(1-\chi_1)(-\Delta - \lambda^2)R_0(\lambda) = R(\lambda)\{(P-\lambda^2)(1-\chi_1) + [\chi_1,\Delta]\}R_0(\lambda)$$
 to write
\begin{equation}\label{e:rzr0z}
 R(\lambda)(1-\chi_1) = \{1-\chi_1 - R(\lambda)[\Delta,\chi_1]\}R_0(\lambda).
\end{equation}
Similarly to \eqref{e:rzr0z} we have
\begin{equation}\label{eq:lco}
 (1-\chi_1)R(z) = R_0(z)\{1-\chi_1 + [\Delta,\chi_1]R(z)\}.
\end{equation}
 We note for later use that this implies that for any $f \in \mathcal H_{c}$ and $\lambda$ in the upper half plane,
\begin{equation}\label{e:rexp}
 R(\lambda) f|_{\mathbb R^2 \setminus \mathcal B}(x) = O(e^{-|x|\im \lambda}), \qquad \text{as } |x| \to \infty.
\end{equation}
Since $R(\lambda) - R(z) = (\lambda^2 - z^2)R(\lambda)R(z)$, we have
\[
 R(\lambda) - R(z)  = (\lambda^2 - z^2)\Big(R(\lambda)\chi_1(2-\chi_1)R(z) + R(\lambda)(1-\chi_1)^2 R(z)\Big),
\]
and inserting  \eqref{e:rzr0z} and (\ref{eq:lco}) gives
\[
\begin{split}
R(\lambda)-R(z)= &(\lambda^2-z^2)\Big(R(\lambda) \chi_1(2-\chi_1)R(z)
\\
&+ \{ (1-\chi_1)-R(\lambda)[\Delta,\chi_1]\} R_0(\lambda)R_0(z)
\{ (1-\chi_1)+[\Delta,\chi_1]R(z)\}\Big).
\end{split}
\]
Plugging in $(\lambda^2 - z^2)R_0(\lambda)R_0(z) = R_0(\lambda)-R_0(z)$,
and putting
\begin{equation}\label{e:k1def}
 K_1 = 1 - \chi_1 +[\Delta,\chi_1] R(z),
\end{equation}
gives
\begin{equation}\label{e:vodevido}
\begin{split}
R(\lambda)-R(z)= &(\lambda^2-z^2)R(\lambda) \chi_1(2-\chi_1)R(z) + \{ 1-\chi_1-R(\lambda)[\Delta,\chi_1]\} (R_0(\lambda)-R_0(z))K_1.
\end{split}
\end{equation}
We now bring the $R(\lambda)$ terms to the left, the remaining terms to the right, and factor, obtaining
\begin{equation}\label{e:vodevid}
 R(\lambda)  (I - K(\lambda)) = F(\lambda),
\end{equation}
where
\begin{align}
 K(\lambda) &= (\lambda^2 - z^2)\chi_1(2-\chi_1)R (z)  - [\Delta, \chi_1](R_0(\lambda) - R_0(z))K_1,\\
 F(\lambda) &= R(z) + (1-\chi_1)(R_0(\lambda) - R_0(z))K_1 . \label{e:fdef}
\end{align}
Here and below we shorten formulas by using notation which displays $\lambda$-dependence but not $z$-dependence for operators other than resolvents.  The identities \eqref{e:vodevido} and \eqref{e:vodevid} are versions of Vodev's resolvent identity \cite[$(5.4)$]{vo14}.

Our first use of Vodev's identity is to prove meromorphic continuation of the resolvent to $z$, using the technique of Section 2 of \cite{cd}. Take $\chi \in C_c^\infty(\mathbb R^2)$ such that $\chi$ is $1$ near the support of $\chi_1$ and multiply \eqref{e:vodevid} on the left and right by $\chi$. That gives $\chi R(\lambda)\chi(I-K(\lambda)\chi) = \chi F(\lambda)\chi$. Observe now that $K(\lambda) \chi$ is compact $\mathcal H \to \mathcal H$, and $\|K(\lambda)\|_{\mathcal H \to \mathcal H} \to 0$ as $\lambda \to z$. Consequently, by the analytic Fredholm theorem, $\chi R(\lambda) \chi =  \chi F(\lambda)\chi(I-K(\lambda)\chi)^{-1}$ continues meromorphically from the upper half plane to $\Lambda$, the Riemann surface of $\log \lambda$.

Thus  \eqref{e:vodevido} and \eqref{e:vodevid} continue to hold for any $z$ and $\lambda$ in $\Lambda$, with $K(\lambda)$ and $K_1$  mapping 
$\mathcal H_{\text{c}}$ to $\mathcal H_{\text{c}}$, and
$R(\lambda)$ and $F(\lambda)$ mapping  $\mathcal H_{\text{c}}$ to $\mathcal D_{\text{loc}}$.

\section{Resolvent expansions with mild growth}\label{s:rs0}

The main result of this section is the proof of Theorem \ref{t:resexp},  an asymptotic expansion for the resolvent near the origin under the mild growth assumption \eqref{e:uhpbd}.

We now begin to use the assumption \eqref{e:uhpbd} on the rate at which the cutoff resolvent norms may grow near zero energy.  Parts of this proof are the same as that of \cite[Theorem 1]{ChDaob}.  That result, for Dirichlet obstacle scattering, is a special case of \eqref{e:no0resser}.

Our first lemma, from \cite{ChDaob}, is based on Vodev's identity \eqref{e:vodevid} and on part of  \cite[Proposition~3.1]{vo99}. To state it, use the free resolvent series \eqref{e:r0series}  and $F(\lambda)$ as in \eqref{e:fdef} to write
\begin{equation}\label{e:fseries}
 F(\lambda) = \sum_{j=0}^\infty \sum_{k=0}^1  F_{2j,k} \lambda^{2j} (\log \lambda)^k =  F_{01} \log \lambda + \tilde F(\lambda),
\end{equation}
where
each $F_{2j,k}$ is bounded $\mathcal H_{c} \to \mathcal D_{\text{loc}}$.  Moreover, if $k\not =0$, then $F_{2j,k}$ has finite rank.

\begin{lem}\label{l:rfk}
Assume \eqref{e:uhpbd}. There is $z_0 > 0$ such that for every $z$ on the positive imaginary axis obeying $0<-iz\le z_0$,  we have
\begin{equation}\label{eq:smsing}\begin{split}
R(\lambda) =   - \frac{\log z}{1-(\log \lambda-\log z)\alpha(\lambda)} \tilde{F}(\lambda)D(\lambda) (w\otimes 1)K_1D(\lambda) + 
  \tilde{F}(\lambda)D(\lambda)\\ 
+\frac{\log \lambda}{1-(\log \lambda-\log z)\alpha(\lambda)}\left(  \left(  \frac{  -1}{2\pi} (1-\chi_1)    +\tilde{F}(\lambda) D(\lambda) w\right) \otimes 1      \right) K_1D(\lambda) 
\end{split}\end{equation}
where $w = \frac 1 {2\pi} \Delta \chi_1$,  and 
\begin{equation}\label{e:dseries}
D(\lambda) = \sum_{j=0}^\infty \sum_{k=0}^j D_{2j,k}\lambda^{2j}(\log \lambda)^k, \quad \alpha(\lambda) 
=  \sum_{j=0}^\infty \sum_{k=0}^j \alpha_{2j,k}\lambda^{2j}(\log \lambda)^k,
\end{equation}
for some  operators $D_{2j,k}\colon \mathcal H_{c} \to \mathcal H_{c}$ and complex numbers $\alpha_{2j,k}$ which depend on $z$ but not on $\lambda$. 
If $k\not =0$ then  $D_{2j,k}$ has finite rank.  The series  converge absolutely, uniformly on sectors near zero. We also have the following variant of Vodev's identity
 \begin{equation}
 \label{eq:vodevmod}R(\lambda)(I-A(\lambda)D(\lambda))= F(\lambda) D(\lambda),
 \end{equation}
where $A(\lambda) = (\log \lambda - \log z) (w \otimes 1)  K_1$.
\end{lem}

\begin{proof}
The identity \eqref{eq:smsing} generalizes identity (2.23) of \cite{ChDaob}, and the other assertions generalize Lemma~2.3 of the same paper. The same proofs work in our setting and we omit them.
\end{proof}

We now derive the form of the resolvent expansions, in terms of  two cases depending on $\alpha_{00}$.
\begin{lem}\label{l:basicre} Let the notation and assumptions of Lemma \ref{l:rfk} hold.  

\noindent 1. If $\alpha_{00} = 0$, then there are operators $B_{2j,k} \colon \mathcal H_{c} \to \mathcal D_{\text{loc}}$ such that 
\begin{equation}\label{eq:smsingexp}
R(\lambda) =  \sum_{j=0}^\infty \sum_{k=0}^{2j+1} B_{2j,k}  \lambda^{2j} (\log \lambda)^{k} =  B_{01}  \log \lambda +   B_{00}  +   B_{23}  \lambda^2 (\log \lambda)^3 + \cdots.
\end{equation}

\noindent 2. If $\alpha_{00} \ne 0$, then there are operators $B_{2j,k}, \ \widetilde B_{2j,-k}  \colon \mathcal H_{c} \to \mathcal D_{\text{loc}}$ such that 
\begin{align}\label{eq:smsingexp2}
R(\lambda)  &=  \sum_{j=0}^\infty  \left( \sum_{k=0}^{j}  B_{2j,k}  (\log \lambda)^k + \sum_{k=1}^{j+1}   \widetilde B_{2j,-k}  (\log \lambda - \log z - \alpha_{00}^{-1})^{-k} \right)\lambda^{2j} \nonumber \\
&=  B_{00}  +   \widetilde B_{0,-1} (\log \lambda - \log z - \alpha_{00}^{-1})^{-1} +  B_{11}   \lambda^2 \log \lambda + \cdots
\end{align}
The series  converge absolutely, uniformly on sectors near zero.  If $k\not =0$, then $B_{2j,k}$ and $\widetilde B_{2j,-k}$ have finite rank.
\end{lem}

\begin{proof}
\noindent 1. If $\alpha_{00} = 0$, then $|(\log \lambda - \log z)\alpha(\lambda)| \to 0$ as $\lambda \to 0$ and, using the series for $\alpha$ from \eqref{e:dseries},  we expand $1/(1-(\log \lambda - \log z)\alpha(\lambda))$ in the series
\[
\sum_{m=0}^\infty (\log \lambda - \log z)^m \alpha(\lambda)^m = \sum_{j=0}^\infty\sum_{k=0}^{2j} a_{2j,k} \lambda^{2j}(\log \lambda)^k.
\]
Inserting this,  and the series for $D$ and $\tilde F$ from \eqref{e:dseries} and \eqref{e:fseries}, into \eqref{eq:smsing} gives the resolvent series expansion \eqref{eq:smsingexp}, with $B_{01}$ of rank at most $1$.  Moreover, $B_{2j,k}$ has finite rank for any $k\not =0$.

\noindent 2. If $\alpha_{00} \ne 0$, then $|(\log \lambda - \log z)\alpha(\lambda)| \to \infty$ as $\lambda \to 0$. As in equation (2.30) of \cite{ChDaob}, we expand $1/(1-(\log \lambda - \log z)\alpha(\lambda))$ in the series
\[
 \sum_{j=0}^\infty \left( \sum_{k=0}^{j-1} b_{2j,k} (\log \lambda)^k + \sum_{k=1}^{j+1} b_{2j,-k} (\log \lambda - \log z - \alpha_{00}^{-1})^{-k} \right)\lambda^{2j}.
\]
Inserting this series and the series \eqref{e:dseries} and \eqref{e:fseries} for $D$ and $\tilde F$ into \eqref{eq:smsing} gives \eqref{eq:smsingexp2}, with all the $B_{2j,k}$, $\widetilde B_{2j,k}$ having finite
rank if $k \not =0$, and $\widetilde B_{0,-1}$ having rank at most $1$.
\end{proof}

We will use the following two lemmas for uniqueness statements when computing $B_{01}$ and $\widetilde B_{0,-1}$, and also to identify the cases $\alpha_{00} = 0$ and $\alpha_{00} \ne 0$ with the resonant $\mcg_0 \ne \{0\}$ and non-resonant $\mcg_0 = \{0\}$ cases respectively. We defer the proofs to the end of the section.

\begin{lem}\label{l:no0eig}
Assume \eqref{e:uhpbd}. Then $0$ is not an eigenvalue of $P$, i.e. $\mcg_{-2}=\{0\}$. \end{lem}

\begin{lem}\label{l:pubd}
Suppose $R(i\varepsilon)$ has a limit $\mathcal H_{\text{c}} \to \mathcal D_{\text{loc}}$ as $\varepsilon\to 0^+$. Then $\mcg_0 = \{0\}$.
\end{lem}

To compute $B_{01}$ and $\widetilde B_{0,-1}$, we expand the equation $(P-\lambda^2)R(\lambda)=I$, for $\im \lambda>0$, using \eqref{eq:smsingexp} or \eqref{eq:smsingexp2}, and compare powers of $\lambda$, $\log \lambda$, and 
$(\log \lambda -\log z-\alpha_{00}^{-1})^{-1}$. This gives the following identities in the sense of operators $\mathcal H_c \to \mathcal D_{\text{loc}}$:
\begin{equation}\label{eq:niceids}
\begin{array}{lll}
P B_{2j,k}=B_{2j-2,k}, & P \widetilde B_{2j,k}=\widetilde B_{2j-2,k}& \text{if $(j,k)\not =(0,0)$,}\\
PB_{00}=I,
\end{array}
\end{equation}
where we understand that $B_{2j,k}=0$ and $\widetilde B_{2j,k}=0$ for the terms not appearing in \eqref{eq:smsingexp} or \eqref{eq:smsingexp2}.

\begin{lem} \label{l:A01} Let the notation and assumptions of Lemma \ref{l:basicre} hold, and suppose $\alpha_{00}=0$.  Then  
\begin{equation}\label{e:a01u}
B_{01} =  -\frac{1}{2\pi} U_0 \otimes U_0, \; \text{in case (A1), or } \; B_{01}=-\frac{1}{2\pi}U_0\otimes \overline{U_0}\; \text{in case (A2)}
\end{equation}
 and $U_0$ is the unique element of $\mathcal G_0$ obeying
\begin{equation}\label{e:uaasy}
 U_0(x) =   1 - \int_{\mathbb R^2}  \frac {x \cdot y}{|x|^2} (K_1D_{00}w)(y)dy 
+ O(|x|^{-2}), \qquad \text{as } |x| \to \infty.
\end{equation}
\end{lem}

\begin{proof}
From \eqref{eq:smsing} we see that the range of $B_{01} $ is spanned by a function $U_0$ given 
by 
\begin{equation}\label{e:1chifdw}
 U_0 = 1-\chi_1 - 2\pi F_{00}D_{00}w.
\end{equation}
Furthermore, $PU_0=0$ follows from \eqref{eq:niceids}.
 Given the expansion \eqref{e:uaasy}, uniqueness of $U_0$ follows from  $\mcg_{-2} = \{0\}$, by Lemma \ref{l:no0eig}. 
Next we show that
\begin{equation}\label{e:f00asy}
 F_{00} D_{00} w(x) = \frac 1{2\pi} \int_{\mathbb R^2}  \frac {x \cdot y}{|x|^2} (K_1D_{00}w)(y)dy + O(|x|^{-2}).
\end{equation}
By definition \eqref{e:fdef},
\begin{equation}\label{e:f00form}
 F_{00} = R(z)\chi +(1-\chi_1)(R_{00} - R_0(z))K_1.
\end{equation}
But the  $R_0(z)$ and $R(z)$ terms are $O(e^{-|z||x|})$ by \eqref{e:r0exp} and \eqref{e:rexp} because $z$ is on the positive imaginary axis. Thus, with $f = K_1D_{00}w,$ \eqref{e:f00asy} follows from the fact that for $|x|$ large enough we have
\begin{equation}\label{e:r00k1}
 R_{00}K_1D_{00}w(x) =  \frac 1 {4\pi}\sum_{m=1}^\infty \frac 1 {m|x|^{2m}} \int_{\mathbb R^2}   (2x\cdot y - |y|^2)^m f(y)\,dy, 
\end{equation}
where we used the $R_{00}$ asymptotic \eqref{eq:R2jk} and the fact that   $\alpha_{00} = \int_{\mathbb R^2} f 
= 0$.  This shows
(\ref{e:uaasy}).

If $P$ is self-adjoint, then writing 
\begin{equation}\label{e:rizuv}
\langle R(i|\lambda|) u, v \rangle - \langle u,  R(i|\lambda|) v \rangle=0, 
\end{equation}
for arbitrary $u$ and $v$ in $\mathcal H_{c}$ and  $|\lambda|>0$ small enough, substituting $R(i|\lambda|) = B_{01} \log |\lambda| + O(1)$, and extracting the coefficient of $\log|\lambda|$, we see that $B_{01}=c\, U_0 \otimes U_0$, for some $c \in \mathbb R$.  

If, instead $P$ and $\mch$ satisfy the hypotheses (A2), then $B_{01}=c \, U_0 \otimes v$ for some $v\in \mch_{loc}$ and some $c \in \Complex$.  Expanding
$R(\lambda)$ as in (\ref{eq:smsingexp}), from the coefficient of $\log \lambda$ in Lemma \ref{l:resA2} we obtain $v=\overline{U_0}$.  

Now we evaluate 
$c$. Let $\phi = [\Delta,\chi_1]\log|x|$. Since  $\chi_1(x) = 0$ if $|x| \ge r_1$,
 the coefficient of $\log \lambda$ in Lemma \ref{l:bpl} yields
\begin{equation}\label{e:ba01phi}
 \bp(B_{01}\phi,\overline{R}_{00}(\bullet, 0))+\bp(B_{00}\phi,\overline{R}_{01}(\bullet, 0))=0.
\end{equation}
Now we evaluate $B_{01}\phi$ and $B_{00}\phi$.
 If $P$ is
self-adjoint, then, by Lemmas \ref{l:bpid} and \ref{l:bdrypair}, we obtain
\[
\langle  \phi,U_0\rangle = \langle \left(P(1-\chi_1)-(1-\chi_1)P\right)\log |\bullet |,U_0\rangle = \bp(\log |\bullet |, U_0)=-2\pi.
\]
Using this in $B_{01} \phi =c\, (U_0 \otimes U_0 )\phi = c\langle  \phi,U_0\rangle U_0$ yields
\begin{equation}\label{e:a01phi}
B_{01} \phi= -2\pi c  \, U_0. 
\end{equation}
A similar argument in case $P$ satisfies hypotheses (A2) also yields \eqref{e:a01phi}. 

To compute $B_{00} \phi$ we observe that, because $P B_{00} \phi = \phi = [\Delta,\chi_1]\log|x| = -\Delta(1-\chi_1)\log|x|$, it follows that $\phi_1 : = B_{00} \phi - (1-\chi_1) \log |x|$ is in the nullspace of $P$. 
Next, by the definition of $F$ \eqref{e:f00form}, we have $F_{00} \psi = O(\log |x|)$ as $|x| \to \infty$, for any $\psi \in \mch_c$. Hence, by the resolvent expansion \eqref{eq:smsing}, we get $B_{00} \phi = O(\log|x|)$ and thus also $\phi_1 = O(\log|x|)$ and so $\phi_1 \in \mcg_{\log}$. Since $U_0 \in \mathcal G_0\setminus\mcg_{-1}$, by Corollary \ref{c:g0glog} it follows that $\phi_1 = O(1)$ as $|x| \to \infty$. Hence
\begin{equation}\label{e:a00phi}
 B_{00} \phi = \log |x| + O(1), \qquad \text{as } |x| \to \infty.
\end{equation}
Using Lemma \ref{l:bdrypair} to evaluate the boundary pairings in \eqref{e:ba01phi}, and plugging in \eqref{e:a01phi}, \eqref{e:a00phi} and the free resolvent asymptotics \eqref{eq:R2jk}, implies that $2\pi(-2\pi c)(\frac {-1} {2\pi}) - 2\pi (1) (\frac {-1}{2\pi})=0$, or $c = \frac {-1}{2\pi}$. 
 \end{proof}

\begin{lem} \label{l:B01} Let the notation and assumptions of Lemma \ref{l:basicre} hold, and suppose $\alpha_{00}\not=0$.  Then  
the operator $\widetilde B_{0,-1}$ from \eqref{eq:smsingexp2} is given by 
\begin{equation}\label{e:b01u}
\widetilde B_{0,-1} =  \frac 1 {2\pi} U_{\log} \otimes  U_{\log} \; \text{in case (A1), or } \; \widetilde B_{0,-1} =  \frac 1 {2\pi}  U_{\log}  \otimes \overline{ U_{\log} } \; \text{in case (A2)}
\end{equation}
where  $ U_{\log} $ is the unique element of $\mathcal G_{\log}$  such that $c_{\log}( U_{\log} ) =1$. Moreover, 
\begin{equation}\label{e:ubasy}
 c_0(U_{\log}) =    \alpha_{00}^{-1} + \log z -  \gamma_0.
\end{equation}
\end{lem}

\begin{proof}
The proof in case (A1) is the same as that of Lemma 2.6 in \cite{ChDaob}, and we give just the following outline: first show that $B_{0,-1}= c   U_{\log}  \otimes  U_{\log} $, then that $ U_{\log}  = 2\pi \widetilde B_{0,-1} w$, and then that $\langle w,  U_{\log}  \rangle_\mch = 1$. Uniqueness of $U_{\log}$ follows from Lemma \ref{l:pubd}

The proof in case (A2) is only slightly different. We have instead $\widetilde B_{0,-1}= c \,  U_{\log}  \otimes \overline{ U_{\log} }$ (just as in the proof of \eqref{e:a01u} above), while $ U_{\log}  = 2\pi \widetilde B_{0,-1} w$ is unchanged, and finally we show $\langle w, \overline{ U_{\log} } \rangle_\mch = 1$ by the same calculation that gives $\langle w,  U_{\log}  \rangle_\mch = 1$ in case (A1).
\end{proof}

We are now ready to complete the proof of Theorem \ref{t:resexp}.

\begin{proof}[Proof of Theorem \ref{t:resexp}]

If $P$ has a zero resonance, then,  by Lemma \ref{l:pubd}, $R(\lambda)$ cannot have a limit as $\lambda \to 0$ along the positive imaginary axis. Hence  by part 2 of Lemma \ref{l:basicre} it follows that $\alpha_{00} = 0$. The conclusion then follows from part 1 of Lemma \ref{l:basicre} and Lemma \ref{l:A01}.

If $P$ has no zero resonance, then there can be no solution $u_A \in \mathcal D_{\text{loc}}$  to $Pu=0$ with the asymptotic \eqref{e:uaasy}. Hence, by part 1 of Lemma \ref{l:basicre} and Lemma \ref{l:A01}, we must have $\alpha_{00} \ne 0$. The conclusion then follows from part 2 of Lemma \ref{l:basicre} and Lemma \ref{l:B01}.
\end{proof}

It remains to prove Lemmas \ref{l:no0eig} and \ref{l:pubd}. We prove Lemma \ref{l:no0eig} using an intermediate step in the proof of Vodev's identity to reduce to a problem of bounding the free resolvent. See Lemma 5 of \cite{schlag} for a related calculation. If $P$ is self-adjoint, then Lemma~\ref{l:no0eig} can alternatively be deduced from Proposition \ref{p:resexpsa}  below.

\begin{proof}[Proof of Lemma \ref{l:no0eig}]

Let $\chi,\chi_0\in C_c^\infty(\Real)$ be $1$ in a neighborhood of $\mathcal{B}$, with $\supp \chi_0$ contained in the set on which $\chi_1$ is $1$.  
We use \eqref{e:rzr0z} with $\lambda =i\kappa$, $\kappa>0$ to write
\begin{equation}\label{eq:split}
 R(i\kappa)= R(i\kappa )\chi_1+ \{ 1-\chi_1-R(i\kappa ) [\Delta,\chi_1]\} R_0(i\kappa )(1-\chi_0).
 \end{equation}
Let $\phi \in \mcg_{-2}$.  Then $R(i\kappa ) \phi =\phi/\kappa^2$.   Applying  \eqref{eq:split} to $\phi $ and multiplying on the left by $\chi$ yields
$$ \kappa^{-2}\chi \phi = \chi R(i\kappa )\chi_1\phi + \chi \{ 1-\chi_1-R(i\kappa ) [\Delta,\chi_1]\} R_0(i\kappa )(1-\chi_0)\phi,$$
or, using our assumption \eqref{e:uhpbd}, 
\begin{equation}\label{eq:dog}
\chi \phi = \kappa^2 \chi \{ 1-\chi_1-R(i\kappa ) [\Delta,\chi_1]\} R_0(i\kappa )(1-\chi_0)\phi +o((\log \kappa)^{-1}).
\end{equation}
It suffices to show that   $\|\tilde{\chi}R_0 (i\kappa) (1-\chi_0)\phi\| = O(\log \kappa)$ as $\kappa \downarrow 0$, for any $\tilde{\chi}\in C_c^\infty(\Real^2)$,
since then by 
\eqref{eq:dog} and \eqref{e:uhpbd} we must have $\chi \phi=o(1)$, or $\phi=0$, proving the lemma.  

Moreover, we can replace $(1-\chi_0)$ with ${\bf 1}_{> \rho}$ for any fixed $\rho>0$, since by the  $R_0$ expansion \eqref{e:r0series} we have $\|\tilde \chi R_0(i\kappa)(1 - \chi_0 -  {\bf 1}_{> \rho})\|_{L^2 \to L^2} = O(\log \kappa)$. Thus, by the resolvent kernel formula \eqref{e:h10ser}, taking $\rho$ large enough that $\tilde \chi$ is supported in $\{x \colon |x| < \rho/2\}$, we see that it is enough to show
\[
 \int_{\{y \colon |y| > \rho\}} H^{(1)}_0(i \kappa|x-y|) \phi(y)\,dy = O(\log \kappa),
\]
uniformly for $x \in \mathbb R^2$ such that $ |x| < \rho/2$. Now write $|x-y| = |y|(1+\varepsilon)$, where $\varepsilon = \varepsilon(x,y)$ obeys $|\varepsilon|\le|x|/|y|<1/2$.   Using  the fact that
\[
 \int_{\{y \colon |y| > \rho\}} H^{(1)}_0(i \kappa|y|) \phi(y)\,dy = 0
\]
since the zero Fourier component of $\phi$ is $0$ for $|y|>\rho$, we 
 see that it is enough to show that
\[
 \int_{\{y \colon |y| > \rho\}} \Big(H^{(1)}_0(i \kappa|y|(1+\varepsilon)) - H^{(1)}_0(i \kappa|y|) \Big) \phi(y)\,dy = O(\log \kappa).
\]
Using $\phi(y) = O(|y|^{-2})$, and passing to polar coordinates, observe that it is enough to prove
\[
 \int_{\rho\kappa}^\infty \Big|H^{(1)}_0(i  s(1+\varepsilon_1)) - H^{(1)}_0(i s) \Big|s^{-1}ds = O(\log \kappa),
\]
provided $|\varepsilon_1|<1/2$.
By the same large argument Bessel function asymptotics as  in the proof of \eqref{e:r0exp}, we have $H_0^{(1)}(is) \lesssim e^{-s}$ as $s \to \infty$, so it is enough to show that there is $\delta>0$ such that 
\[
 \int_{\rho\kappa}^\delta \Big|H^{(1)}_0(i  s(1+\varepsilon_1)) - H^{(1)}_0(i s) \Big|s^{-1}ds = O(\log \kappa).
\]
For that use the fact that, by \eqref{e:h10ser}, $H_0^{(1)}(is) \sim (2i/\pi) \log s$ as $s \to 0$ to write 
\[
  \Big|H^{(1)}_0(i  s(1+\varepsilon_1)) - H^{(1)}_0(i s) \Big| \sim \frac  2 \pi  |\log (1 + \varepsilon_1)| < \frac 2 \pi \log 2.
\]
\end{proof}

\begin{proof}[Proof of Lemma \ref{l:pubd}]We first give the proof for $P$ self-adjoint.
Let $U \in \mcg_0$. Take $\phi \in \mathcal H_{\text{c}}$,
and take  $r_1$ such that $\phi(x)=0$ when $|x|\ge r_1$. Let 
$R(0)= \lim_{\varepsilon \to 0^+}R(i\varepsilon)$.
Then, by Lemma \ref{l:bpid}, 
\begin{equation}\label{e:uphi}
0 =\langle  R(0)\phi,P U \rangle  = \langle \phi, U \rangle  - \bp(R(0)\phi, U)
\end{equation}
Now observe that, by \eqref{eq:lco} with  $\chi_1$ such that $\chi_1 = 1$ near the support of $\phi$,  with  $z \to 0$ along the positive imaginary axis, and substituting the free resolvent series \eqref{e:r0series}, we have that $(1-\chi_1)R(0)\phi$ is given by the limit
\[
\lim_{\varepsilon \to 0^+} \log(i\varepsilon)R_{01}\{1-\chi_1 + [\Delta,\chi_1]R(i\varepsilon)\}\phi + R_{00}  \{1-\chi_1 + [\Delta,\chi_1]R(0)\}\phi.
\]
But the fact that this limit  exists implies that 
\begin{equation}\label{e:1chi1phi0}
\int ((1-\chi_1+[\Delta,\chi_1]R(0)) \phi)(x)dx= 0,
\end{equation}
and  that  
$\lim_{\varepsilon \to 0^+} \log (i\varepsilon) R_{01} \Big(1-\chi_1+[\Delta,\chi_1](R(i\varepsilon)-R(0))\Big) \phi$ exists, yielding 
\begin{equation}\label{e:r00phi1}
\begin{split}
(1-\chi_1)R(0)\phi = \lim_{\varepsilon \to 0^+} \log (i\varepsilon)  &R_{01} [\Delta,\chi_1](R(i\varepsilon)-R(0)) \phi + R_{00}(1-\chi_1+[\Delta,\chi_1]R(0)) \phi .
\end{split}
\end{equation}
Now we use \eqref{e:1chi1phi0}  
and insert the $R_{00}$ series \eqref{eq:R2jk} into \eqref{e:r00phi1} to get
\[
(1-\chi_1)R(0)\phi(x) =  c_\phi+ \frac 1 {4\pi}\sum_{m=1}^\infty \frac 1 {m|x|^{2m}} \int_{\mathbb R^2}  (2x\cdot y - |y|^2)^m f(y)\,dy,
\]
with $f= (1-\chi_1+[\Delta,\chi_1]R(0))\phi$
and  $$c_\phi= -\frac{1}{2\pi}\lim_{\varepsilon \to 0^+} \int \Big( \log (i\varepsilon)  [\Delta,\chi_1](R(i\varepsilon)-R(0))\Big) \phi \; dx .$$
In particular, $R(0)\phi$ is harmonic and bounded for large $|x|$. By Lemma \ref{l:bdrypair}, $\bp(R(0)\phi, U) =0$, and hence \eqref{e:uphi} yields $\langle \phi, U \rangle =0$. Since this is true for any $\phi \in \mathcal H_{c}$, it follows that $U=0$.  

If $P$ is not self-adjoint but satisfies the alternate conditions (A2), then instead of (\ref{e:uphi}) we write
$$0=\langle  R(0)\phi,\overline{PU}\rangle = \langle  R(0)\phi, P^* \bar U \rangle = \langle \phi, \bar U  \rangle  - \bp(R(0)\phi,\bar U).$$
The rest of the proof is essentially identical, giving $\langle \phi, \bar U\rangle=0$ for all $\phi \in \mch_{c}$, and hence $U=0$.
\end{proof}

\section{General resolvent expansions}\label{s:fsrb}
Here we begin our study of resolvent expansions without the mild growth hypothesis (\ref{e:uhpbd}).  Corollary \ref{c:nspchar} to these initial results gives us a sufficient condition for (\ref{e:uhpbd})
in terms of the nullspace of $P$:  if $\mcg_{-1} = \{0\}$, i.e. if $P$ has no $p\,$-resonance or eigenvalue at $0$, then \eqref{e:uhpbd} holds. In the self-adjoint case, the converse is also true by Theorem \ref{t:nre}, or more specifically by Lemma \ref{l:pressing}.

We emphasize that the results of this section do not require that $P$ be self-adjoint, only that it satisfies our general black-box hypotheses, including either (A1) or (A2).

\begin{prop} \label{p:rexpbasic}There are $j_0\in \Natural_0$, $ k_0(j)\in \Natural_0$ and $B_{2j,k}:\mch_c\rightarrow \mathcal{D}_{\loc}$ such that 
\begin{equation}\label{eq:expwo}
R(\lambda)=\sum_{j =- j_0}^\infty\sum_{k = -\infty}^{k_0(j)} B_{2j,k} (\log \lambda)^k \lambda^{2j}.
\end{equation}
The series \eqref{eq:expwo} converges absolutely, uniformly on sectors near zero.
Moreover, if $k\not =0$, then $B_{2j,k}$ has finite rank.
\end{prop}

\begin{proof}
We use Vodev's identity $R(\lambda)(I-K(\lambda))=F(\lambda)$ from \eqref{e:vodevid}.  
First note that if we choose $\chi_2\in C_c^\infty(\Real^2)$ to be $1$ on the support of $\chi_1$, then $(I-K(\lambda)(1-\chi_2))^{-1}=I+K(\lambda)(1-\chi_2)$, since $(1-\chi_2)K(\lambda)=0$.  Using $I-K(\lambda)=(I-K(\lambda)(1-\chi_2))(I-K(\lambda)\chi_2)$
yields
$$R(\lambda)=F(\lambda) (I-K(\lambda)\chi_2)^{-1}(I+K(\lambda)(1-\chi_2)),$$
so that it 
remains to understand the inverse of $I-K(\lambda)\chi_2$. 
Note that, by Section \ref{s:freeres}, near $\lambda=0$ the free resolvent $R_0(\lambda)$  can be written in the form 
\begin{equation}\label{e:h1h2}
 (\log \lambda) H_1(\lambda) +H_2(\lambda), \quad \begin{array}{l}  H_1, \ H_2, \text{ holomorphic and even in }\lambda, \\ \text{and } \partial_\lambda^{m} H_1(0) \text{ finite rank for all }m. \end{array}\end{equation}
Hence each of $K(\lambda)$ and $F(\lambda)$ can be written in the form \eqref{e:h1h2} as well.
Since $K(\lambda)\chi_2$ is a compact operator  and $K(z)=0$, \eqref{eq:expwo} follows from 
the Hahn-holomorphic Fredholm theorem \cite[Theorem 4.1]{MuSt}.  

For the reader's convenience we give the argument, which is 
very similar to that used in the proof of the usual analytic Fredholm theorem, e.g. \cite[Theorem VI.14]{RS}. 
Set 
$$A_1(\lambda):= (1/2\pi)(\log \lambda -\log z)(\Delta\chi_1\otimes 1)K_1\chi_2.$$
 Since $\lim_{\lambda \to 0} K(\lambda)\chi_2 - A_1(\lambda)$ is compact,  we can find $\lambda_0>0$ small enough, a finite rank operator $A_2$, and a 
compact operator $K_2(\lambda)$, such that 
$$K(\lambda)\chi_2=A_1(\lambda)+A_2+K_2(\lambda), \qquad \text{for $\lambda\in \Lambda$, $|\arg \lambda|<\varphi_0,\; |\lambda|<\lambda_0$},$$
and $\|K_2(\lambda)\|\leq 1/2$ in this region.
Moreover, $K_2(\lambda)$ can be written in the form \eqref{e:h1h2}. Then $I-K_2(\lambda)$ is invertible with bounded 
inverse in this region, and the inverse can be written
$$(I-K_2(\lambda))^{-1}=\sum_{m=1}^\infty (K_2(\lambda))^m=\sum_{j\geq 0}\sum_{0\leq k \leq j} D_{2j,k}(\log \lambda)^k \lambda^{2j}$$
for some operators $D_{2j,k}:\mch_c\rightarrow \mch_{\loc}$, which have finite rank when $k>0$.

We conclude that $I-K(\lambda)\chi_2= (I-A_3(\lambda))(I-K_2(\lambda))$, where
$A_3(\lambda)= \left(A_1(\lambda)+A_2\right))(I-K_2(\lambda))^{-1}$.  Since $A_1(\lambda)$ and $A_2$ are each finite rank, $A_3(\lambda)$ is of finite rank and
$I-A_3(\lambda)$ can be inverted essentially by Cramer's rule, resulting in an operator of the form $I+G(\lambda)$, where $G(\lambda)$ is of finite rank and has an expansion of the 
type on the right hand side of \eqref{eq:expwo}.   The claim about the finite rank of $B_{2j,k}$ for $k\not =0$ follows from how a nonzero power of $\log \lambda$ can arise
during the construction: as a coefficient of $\log \lambda \;\lambda^{2j'}$ in the expansions of $F$, $K$, or $K_2$,
each of which is of finite rank, or in inverting $I-A_3(\lambda)$, which differs from the identity by a finite rank operator.
\end{proof}

\noindent\textbf{Remark.}  The main ingredients of the proof of this proposition are Vodev's identity and a variant of the analytic Fredholm theorem. 
This means that the proof of Proposition \ref{p:rexpbasic} does not need the full strength of assumptions (A1) or (A2).   It does, however, require
that $K(\lambda)$ is a compact operator for some $\lambda$ with $0<\arg \lambda<\pi$.  For this, it suffices if $u \mapsto ((P+z_0)^{-1}u)|_{\mathcal B}$ is compact for some $z_0$ with $\im z_0\not = 0.$  Hence the results of Proposition \ref{p:rexpbasic} hold for a larger class of 
non-selfadjoint operators, including the case of a finite set of subwavelength resonators as in \cite{ADH}.

\begin{lem}\label{l:wsing} With the assumptions and notation of Proposition \ref{p:rexpbasic}, if $j_0>0$ and $\phi \in \mch_c$, then $B_{-2j_0,k} \phi \in \mcg_{-1}$ for all $k\in\Integers$ with $k\leq k_0(-j_0)$.
\end{lem}
\begin{proof}
To simplify notation,  let  $k^\sharp=k_0(-j_0)$, and for each $k \le k^\sharp$ let $\psi_k = B_{-2j_0,k}\phi$. Since $P\psi_k=0$ follows from the coefficient of $\lambda^{-2j_0}(\log \lambda)^{-k}$ in $(P-\lambda^2)R(\lambda)\phi=\phi$, it is enough to show that, for all $k \le k^\sharp$, we have
\begin{equation}\label{e:b2j0ks}
\psi_k(x)=O(|x|^{-1}), \qquad \text{as } |x| \to \infty.
\end{equation} 
From the coefficient of $z^{-2j_0}(\log z)^{k^\sharp+1}$ in \eqref{eq:lco}, and using $R_{01} = - \frac 1 {2\pi} 1 \otimes 1$, we obtain 
\begin{equation}\label{eq:fworst}
\langle [\Delta,\chi_1]\psi_{k^\sharp},1\rangle=0,
\end{equation}
and from the coefficient of $z^{-2j_0}(\log z)^{k}$ with $k\leq k^\sharp$ we obtain
\begin{equation}\label{eq:sworst}
\psi_k= R_{00}[\Delta,\chi_1]\psi_k + R_{01} [\Delta,\chi_1]\psi_{k-1} +O(|x|^{-N}) = O(\log|x|).
\end{equation}
Plugging  the formula (\ref{eq:R2jk}) for $R_{00}$ and \eqref{eq:fworst} into \eqref{eq:sworst} with $k=k^\sharp$, and applying Lemma~\ref{l:Ucommint}, yields 
\begin{equation}\label{eq:tworst} \psi_{k^\sharp}(x)= c_{\log}(\psi_{k^\sharp-1}) +O(|x|^{-1}).
\end{equation}
By Corollary \ref{c:g0glog}, we cannot have both $\psi_{k^\sharp} \in \mcg_0\setminus\mcg_{-1}$ and $\psi_{k^\sharp-1} \in \mcg_{\log}\setminus\mcg_0$. Hence $ c_{\log}(\psi_{k^\sharp-1}) = 0$, and we obtain 
$\psi_{k^\sharp}(x)=O(|x|^{-1})$ and $\psi_{k^\sharp-1}(x)=O(1).$

An inductive argument which repeats the steps above proves \eqref{e:b2j0ks} for all $k\leq k^\sharp$.
\end{proof}

Lemma \ref{l:wsing} has as an immediate corollary:
\begin{cor} \label{c:nspchar} Suppose $P$ satisfies the black-box hypotheses, including either (A1) or (A2).  If $\mcg_{-1} = \{0\}$, then the resolvent estimate
(\ref{e:uhpbd}) holds.
\end{cor}

\section{Resolvent expansions without mild growth} \label{s:renobd}
In this section we study the more challenging problem of understanding the expansion of $R(z)$ near $0$ without the mild growth assumption \eqref{e:uhpbd}.
  In order to make this more manageable,
we shall assume throughout this section that $P$ satisfies hypothesis (A1); in particular, $P$ is self-adjoint.

\begin{prop}\label{p:resexpsa}
When $P$ is self-adjoint, \eqref{eq:expwo} can be refined to 
\begin{equation}\label{eq:reb}
R(\lambda)= \sum_{k=-\infty}^0 B_{-2,k} (\log \lambda)^k \lambda^{-2} + \sum_{j=0}^\infty \sum_{k=-\infty}^{k_0(j)} B_{2j,k} (\log \lambda)^k \lambda^{2j}, \ B_{-2,0} = -\Pe.
\end{equation}
\end{prop}

\begin{proof}
By self-adjointness,  $\im \langle (P\mp i r^2) \phi, \phi\rangle = \mp r^2 \|\phi\|^2$ if $\phi \in  \mathcal H_{\text{c}}$, $r>0$. Hence
$\| R(e^{i\pi/4}r)\|_{\mch \to \mch} \leq r^{-2}$.   Combining  with Proposition \ref{p:rexpbasic} gives \eqref{eq:reb},  except for $B_{-2,0} = -\Pe$.
 
It remains to prove  $B_{-2,0} = -\Pe$. First observe that  $\|B_{-2,0}\|_{\mch \to \mch}\le 1$, because if $\phi, \ \psi \in \mch_c$, then $|\langle B_{-2,0}\phi,\psi\rangle| =  \lim_{r \to 0^+}|\langle r^2 R(e^{i\pi/4}r)\phi,\psi\rangle|    \le \|\phi\| \|\psi\|$, by  \eqref{eq:reb} and $\| R(e^{i\pi/4}r)\| \leq r^{-2}$. Next, the coefficient of $\lambda^{-2}$ in $(P- \lambda^2)R(\lambda) = I$ gives $PB_{-2,0}=0$ and thus $\Pe B_{-2,0} =B_{-2,0}.$

The coefficient of $\lambda^0(\log \lambda)^0$ in  $(P- \lambda^2)R(\lambda) = I$ gives $ P B_{00} - B_{-2,0} = I$, and applying $\Pe$ to both sides gives $\Pe B_{-2,0} = - \Pe$ once we show that $\Pe P B_{00} = 0$.

It remains to prove $\Pe P B_{00} = 0$. We deduce this from  $\phi \in \mch_c$, $\psi \in \mcg_{-2}$ $\Longrightarrow$ $\langle P B_{00} \phi, \psi \rangle = \langle  B_{00} \phi, P \psi \rangle =0 $ and to check the first equals sign we use Lemma \ref{l:bpid} to reduce the problem to proving $\lim_{r_1 \to \infty}\bpr(B_{00}\phi,\psi) = 0$. Now, by the coefficient of $z^0(\log z)^0$ in  \eqref{eq:lco} we have
\begin{equation}\label{eq:B00}
(1-\chi_1)B_{00} =    R_{00}( 1-\chi_1)   + \sum_{j=0}^1\sum_{k=0}^1 R_{2j,k} [\Delta,\chi_1]B_{-2j,-k}.
\end{equation}
Using the $R_{2j,k}$ formulas in \eqref{eq:R2jk}, the lack of $l \ge 0$ modes in the  expansion \eqref{e:fcexp}  of  $B_{-2,-1} \phi$  (from Lemma \ref{l:wsing}), and the lack of $l \ge -1$ modes in the expansion of   $B_{-2,0} \phi$ (from $\Pe B_{-2,0}=B_{-2,0}$), we see that $B_{00} \phi = O(|x|)$ and $\partial_r B_{00}\phi = O(1)$ as $|x| \to \infty$, and hence from the formula \eqref{eq:bdrypair} for $\bpr$ we have $\bpr(B_{00}\phi,\psi) = O(r_1^{-1})$,  as desired.
\end{proof} 

Like $B_{-2,0}$, all terms $B_{2j,k}$ we compute below have the form 
\begin{equation}\label{e:b2jksa}
 B_{2j,k } = \sum_{m=1}^{M} c_m U_m \otimes U_m, \qquad \text{for some } M \in \mathbb N, \ c_m \in \mathbb R, \ U_m\in \mathcal D_{\loc}.
\end{equation}
Note that  if  $k=-1$ or $k_0(j)$, then  the $\kappa^{2j}(\log \kappa)^{k}$ coefficient of $\langle R(i\kappa) \phi, \psi \rangle = \langle \phi , R(i\kappa)\psi \rangle$ for $\kappa>0$ gives $\langle B_{2j,k} \phi, \psi \rangle = \langle \phi , B_{2j,k} \psi \rangle$, and hence \eqref{e:b2jksa} follows provided we have additionally $k \ne0$.

In Section \ref{s:pres} we show that $B_{-2,k}$ with $k \le -1$ is nonzero if and only if $\mcg_{-1} \ne \mcg_{-2}$, and describe such $B_{-2,k}$ in terms of elements of $\mcg_{-1}\setminus \mcg_{-2}$  (i.e. $p\,$-resonant states). In Section \ref{s:sres}, we show that $B_{0,1}$ is nonzero if and only if $\mcg_0 \ne \mcg_{-1}$ or $\mcg_{-2} \ne \mcg_{-3}$, and describe it in terms of elements of $(\mcg_0 \setminus \mcg_1)\cup (\mcg_{-2}\setminus \mcg_{-3})$ (i.e. $s$-resonant states and eigenfunctions of slowest  decay). We also show that $B_{0,k} = 0$ for $k \ge 2$ always. In Section \ref{s:presshift} we simplify the leading negative powers of $\log \lambda$.

\subsection{Contributions of \textit{p}-resonances}\label{s:pres} 

The main result of this section is the following proposition, which computes $B_{-2,-1}$.  We use the Fourier series notation of \eqref{e:fcexp}.

\begin{prop} \label{p:b-2-1} Let $M =\dim(\mcg_{-1}/\mcg_{-2})$. If $M=0$ then $B_{-2,-k}=0$ for all $k \ge 1$. Otherwise, $M \le 2$ and there exist $\{w_m\}_{m=1}^M$ an orthormal set in $\mathbb C^2$ and corresponding $U_{w_m} \in \mcg_{-1}$ with $v_{-1}(U_{w_m})=w_m$ such that 
\[
 B_{-2,-1} = \frac 1 \pi \sum_{m=1}^{M} U_{w_m} \otimes U_{w_m}.
\]
 
\end{prop}

Before proving Proposition \ref{p:b-2-1}, we prove several lemmas and identities. We begin by using Vodev's identity \eqref{e:vodevido} with $\lambda=z e^{in\pi}$, followed by  $R_0(z e^{in\pi})-R_0(z)=\sum_{j=0}^\infty in\pi R_{2j,1}z^{2j}$, to get
\begin{equation}\label{eq:vsym}
\begin{split}
R(&z e^{in\pi})-R(z)= i n \pi \Big( 1-\chi_1-R(z e^{in\pi})[\Delta,\chi_1] \Big) \Big(\sum_{j=0}^\infty R_{2j,1} z^{2j}\Big)\Big(  1-\chi_1+[\Delta,\chi_1]
R(z) \Big) .
\end{split}
\end{equation}

In the next lemma we extract coefficients corresponding to negative powers of $\log z$ from \eqref{eq:vsym}.

\begin{lem}\label{l:vodidnegpwr}
Let $-1 \le j $ and $1 \le k' \le k$. Then
\begin{equation}\label{e:b2j2kcomm}
B_{2j,-k}= \sum_{j_1+j_2+j_3 = j} B_{2j_1,-k'} [\Delta,\chi_1]R_{2j_2,1}[\Delta,\chi_1] B_{2j_3,k' -k -1}
\end{equation}
For $j=-1$, \eqref{e:b2j2kcomm} simplifies to
\begin{equation}\label{e:b-2-1comm}
B_{-2,-k} =B_{-2,-k'} [\Delta,\chi_1]R_{21}[\Delta,\chi_1] B_{-2,k'-k-1},
\end{equation}
In particular,  $B_{-2,-k}\mch_c \subset B_{-2,-k'}\mch_c$. 
\end{lem}

\begin{proof} To prove \eqref{e:b2j2kcomm}, use $a^{-k}-b^{-k}=(b-a)(a^{-k}b^{-1}+\cdots+a^{-1}b^{-k})$ to extract the coefficient of $z^{2j} (\log z + in \pi)^{-k'}(\log z)^{k'-k-1}$ from \eqref{eq:vsym}. 

To get \eqref{e:b-2-1comm}, it is enough to show that if $j=-1$, then all terms in \eqref{e:b2j2kcomm} with $j_2=0$  vanish. Such terms have either $j_3=-1$ or $j_1=-1$. If $j_3=-1$, we have $R_{01} [\Delta,\chi_1]B_{-2,k-k'-1}=0$, by Lemma~\ref{l:Ucommint} and the fact that $B_{-2,k-k'-1}$ maps $\mch_c \to \mcg_{-1}$ by Lemma \ref{l:wsing}. 
If $j_1=-1$, we have similarly $B_{-2,-k'} [\Delta,\chi_1]R_{01}=0$, because $B_{-2,-k}^*$  maps $\mch_c \to \mcg_{-1}$ since $P$ is self-adjoint.
\end{proof}
 
We shall also need to compute $R_{21}[\Delta,\chi_1]U_{w}$.
Since $(1-\chi_1)U_w$ has no $0$ Fourier modes, 
$$\langle (1-\chi_1)U_{w}, \Delta R_{21}(x,\bullet) \rangle_{|x|<r_1} = \langle (1-\chi_1)U_{w}, \frac 1 {2 \pi}\rangle_{|x|<r_1} =  0$$ so that by 
Lemma \ref{l:sGreens} $\langle[\Delta,\chi_1]U_{w}, R_{21}(x,\bullet)\rangle = \bp(U_{w},R_{21}(x,\bullet))$.
Then 
\begin{equation}\label{eq:R21U}
\begin{split}
(R_{21}[\Delta,\chi_1]U_{w})(x)&=  \langle[\Delta,\chi_1]U_{w}, R_{21}(x,\bullet)\rangle = \bp(U_{w},R_{21}(x,\bullet)) =-\bp(U_{w},  \frac{x\cdot \bullet}{4\pi})  = -\tfrac{1}{2} w\cdot x,
\end{split}
\end{equation} 
where we used the formula \eqref{eq:R2jk} for $R_{21}$, Lemma \ref{l:bdrypair} to compute  the
boundary pairing, and again the fact that $U_w$ has no $0$ Fourier modes for $|x|$ big enough.

\begin{lem}\label{l:-2,k}
 If $\phi \in \mch_c$ and $B_{-2,-1}\phi \ne 0$, then $B_{-2,-1} \phi  \in \mcg_{-1}\setminus \mcg_{-2}$.
\end{lem}
\begin{proof}
By
Lemma \ref{l:wsing},  $B_{-2,-1} \phi\in \mcg_{-1}$, so we must show that $B_{-2,-1} \phi \in \mcg_{-2}$ implies  $B_{-2,-1} \phi=0$.

By \eqref{e:b-2-1comm}, we have
\begin{equation}\label{eq:B2-1s}
  B_{-2,-1}=  B_{-2,-1}[\Delta,\chi_1]R_{21}[\Delta,\chi_1]B_{-2,-1}.
\end{equation}
But if  $B_{-2,-1} \phi \in \mcg_{-2}$ then $R_{21}[\Delta,\chi_1]B_{-2,-1} \phi=0$ by  \eqref{eq:R21U}, and so  \eqref{eq:B2-1s} implies $B_{-2,-1} \phi=0$.
\end{proof}
 

Now we show that if $\mcg_{-1} \ne \mcg_{-2}$, then there is a nontrivial term $B_{-2,-1}$ in the expansion  (\ref{eq:reb}).  
\begin{lem}\label{l:pressing}
We have $\rank B_{-2,-1} = \dim (\mcg_{-1} / \mcg_{-2})$.
\end{lem}

\begin{proof}
First, Lemma \ref{l:-2,k} implies that $\rank B_{-2,-1} \le \dim (\mcg_{-1} / \mcg_{-2})$. 

To prove that $\dim (\mcg_{-1} / \mcg_{-2}) \le \rank B_{-2,-1}$, we will construct a linear map $L\colon \mcg_{-1} \to B_{-2,-1}\mch_c$ with $\ker L \subset \mcg_{-2}$. Let $U_{w_0} \in \mcg_{-1}$ and $\phi \in \mch_c$ with $\Pe\phi=0$.
 Applying Lemma \ref{l:bpid} with $r_1$ large enough that $\phi(x) = 0$ when $|x| \ge r_1$, and using $PB_{00}=I-\Pe$ (from the $\lambda^0$ coefficient of $(P-\lambda^2)R(\lambda)=I$), gives
 \begin{equation}\label{eq:pstatebp}
0 = \langle P U_{w_0}, B_{00}\phi\rangle  = \bp(U_{w_0},B_{00}\phi)  +  \langle U_{w_0}, \phi \rangle.
 \end{equation}  
To compute $\bp(U_{w_0},B_{00}\phi)$, note that  the $R_{2j,k}$ formulas \eqref{eq:R2jk} imply that as $|x| \to \infty$ we have $R_{01}[\Delta,\chi_1]B_{0,-1}\phi =O(1)$ and $ R_{00}\left( 1-\chi_1+[\Delta,\chi_1]B_{00}\right)\phi=O(\log |x|)$. Furthermore $ B_{-2,0} \phi = -\Pe \phi =0$ by our choice of $\phi$.
Thus, with  $U_{w_1} = B_{-2,-1}\phi$, the formula \eqref{eq:B00} for $B_{00}$  and \eqref{eq:R21U}  yield
\[
B_{00}\phi =R_{21} [\Delta,\chi_1]B_{-2,-1}\phi + O(\log |x|) = - \tfrac 12 w_1 \cdot x + O(\log|x|), 
\]
Hence, by Lemma \ref{l:bdrypair}, $\bp(U_{w_0},B_{00}\phi) = - \pi w_0 \cdot  w_1$. Plugging into \eqref{eq:pstatebp} gives
\begin{equation}\label{e:z0z1}
 w_0 \cdot w_1 = \tfrac 1 \pi \langle U_{w_0}, \phi \rangle.
\end{equation}
Now given $U_{w_0} \in \mcg_{-1}$ with $w_0 \ne 0$, applying \eqref{e:z0z1} with  $\phi_{w_0} = \chi w_0\cdot x$, for some $\chi \in C_c^\infty(\mathbb R^2 \setminus \mathcal B)$, $\chi \ge 0$, $\chi \not \equiv 0$, yields  $w_0 \cdot w_1 > 0$. Thus, defining $L U_{w_0} = B_{-2,-1} \phi_{w_0}$ gives $\ker L \subset \mcg_{-2}$ as desired. 
 \end{proof}

\begin{proof}[Proof of Proposition \ref{p:b-2-1}]
 By
Lemma \ref{l:-2,k}, $B_{-2,-k}=0$ if $M  = 0$ and $k \ge 1$.

It remains to compute $B_{-2,-1}$ when $M \ge 1$. By \eqref{e:b2jksa} and  Lemma \ref{l:pressing},
\[
 B_{-2,-1} = \textstyle\sum_{m,m'=1}^{M} c_{m,m'}U_{w_m}\otimes U_{w_m'},
\]
for some constants $c_{m,m'}$. Further, by linearity, we may take the $w_m$ orthonormal. Using \eqref{eq:R21U} and its consequence $\langle [\Delta,\chi_1]w \cdot \bullet,U_{w'}\rangle = \bp(w \cdot \bullet, U_{w'}) = - 2 \pi w \cdot w'$ in 
  \eqref{eq:B2-1s} gives
$$\textstyle\sum_{m,m'=1}^{M} c_{m,m'}U_{w_m}\otimes U_{w_m'}=\pi  \textstyle\sum_{m,m',m''=1}^{M} c_{mm'}c_{m'm''} U_{w_m}\otimes U_{w_{m''}}.$$
Thus if we denote by $C$ the $M \times M$ matrix whose entries are the $c_{mm'}$, we have $\pi C^2=C$.   Since $C$ must have rank $M$ by Lemma \ref{l:pressing}, this means $\pi C=I$.
\end{proof}

\subsection {Contributions of \textit{s}-resonances}\label{s:sres}

We next compute the $B_{0k}$.  The main results of this section are Lemma  \ref{l:0k>1}, which shows that $B_{0k}=0$ for $k \ge 2$, and Proposition \ref{p:b01form}, which computes $B_{01}$ and shows that it is nontrivial  if and only if $P$ has an $s$-resonant state or a zero eigenfunction which decays more slowly than $|x|^{-3}$ as $|x|\rightarrow \infty$, i.e. if and only if $\mcg_0 \ne  \mcg_1$ or $\mcg_{-2} \ne \mcg_{-3}$.

\begin{lem} \label{l:B2kp} We have $B_{0k}\mch_c \subset \mcg_{-1}$ for $k \ge 2$ and $B_{01}\mch_c \subset \mcg_{0}$.
\end{lem}
\begin{proof} The proof is very similar to that of Lemma \ref{l:wsing}. Let $k^\sharp = k_0(0)$, and assume $k^\sharp \ge 1$ as otherwise there is nothing to prove. Let $\phi \in \mch_c$ and $\psi_k = B_{0k}\phi$ for $k = 1,\dots,k^\sharp$. From the coefficient of $(\log \lambda)^k$ in $(P-\lambda^2)R(\lambda) = I$, $P\psi_k=0$. The coefficient of $(\log z)^{k}$   in \eqref{eq:lco} yields $\psi_k \in \mcg_{\log}$,
\begin{equation}\label{eq:b0k}
\psi_k=  R_{00}[\Delta, \chi_1]\psi_k + c_{\log}(\psi_{k-1})+\delta_{1k} R_{01}(1-\chi_1)\phi +O(|x|^{-N}),
\end{equation}
where we used   $R_{21}[\Delta, \chi_1]B_{-2,0}=0$  from \eqref{eq:R21U}, Lemma \ref{l:Ucommint}, and $\delta_{1k}$ is the Kronecker delta.

By \eqref{eq:b0k} with $k=k^\sharp+1$, we obtain $c_{\log}(\psi_{k^\sharp}) = 0$ and hence $\psi_{k^\sharp} \in \mcg_0$. If $k^\sharp=1$ we are done.

Suppose now $k^\sharp \ge 2$. By \eqref{eq:b0k} with $k= k^\sharp$, and using $\psi_{k^\sharp} \in \mcg_0$, Lemma \ref{l:Ucommint}, and the formula \eqref{eq:R2jk} for $R_{00}$, we obtain $\psi_{k^\sharp} = c_{\log}(\psi_{k^\sharp-1}) +O(|x|^{-1})$. By Corollary \ref{c:g0glog}, we cannot have both $\psi_{k^\sharp} \in \mcg_0\setminus\mcg_{-1}$ and $\psi_{k^\sharp-1} \in \mcg_{\log}\setminus\mcg_0$. Hence $ c_{\log}(\psi_{k^\sharp-1}) = 0$, giving 
$\psi_{k^\sharp} \in \mcg_{-1}$ and $\psi_{k^\sharp-1}\in \mcg_0.$

An inductive argument which repeats the steps above completes the proof.
\end{proof}

\begin{lem}\label{l:0k>1}
 If $k \ge 2$, then $B_{0k}=0$.
\end{lem}
\begin{proof}
As in the proof of Lemma \ref{l:vodidnegpwr}, the $(\log z + in \pi)^{k-2}\log z$  coefficient  
 of  \eqref{eq:vsym} gives  $
B_{0k} =  -B_{0,k-2} [\Delta,\chi_1]R_{01}[\Delta,\chi_1] B_{01} + \delta_{k2} \{\Pe  [\Delta,\chi_1]R_{21} + (1 - \chi_1)R_{01}\}[\Delta,\chi_1] B_{01}$, where $\delta_{k2}$ is the Kronecker delta.
But Lemmas~\ref{l:Ucommint} and \ref{l:B2kp} imply  $R_{01}[\Delta,\chi_1] B_{01}=0$, and $\Pe [\Delta,\chi_1]R_{21} = 0$ by \eqref{eq:R21U}.
\end{proof}

Now that we know $B_{01}$ is the only possible nontrivial $B_{0k}$ with $k \ge 1$, it remains to compute it. 

\begin{prop}\label{p:b01form} There exists $U_0 \in \mcg_0$ such that 
\begin{equation}\label{e:b01form}
B_{01} = - \frac 1 {2\pi} U_0 \otimes U_0 -\frac{\rho^2}{2}\mathcal P_e {\bf 1}_{> \rho} \Pi_{-2} {\bf 1}_{> \rho} \mathcal P_e,
\end{equation} 
where  $\Pi_{-2}$ denotes projection onto the $l= -2$ modes in \eqref{e:fcexp}, $ {\bf 1}_{> \rho}$  is the characteristic function  of   $\{x \colon |x| > \rho\}$ with  $\mathcal B \subset \{x \colon |x| > \rho\}$. 
Moreover, $U_0 = 0$ if $\mcg_0 \setminus \mcg_{-1} = \emptyset$, and $c_0(U_0) = 1$ otherwise.
\end{prop}

\begin{proof} 1. From the coefficient of $(\log z + in\pi)^0(\log z)^0 z^0$ in \eqref{eq:vsym}, we get
\[\begin{split}
 B_{01}&= B_{01,0} + B_{01,-2}, \qquad \text{ where}  \\ B_{01,0}&=(1-\chi_1-B_{00}[\Delta,\chi_1]) R_{01}(1-\chi_1+[\Delta,\chi_1]B_{00}), \qquad B_{01,-2}=
    - \Pe[\Delta,\chi_1]R_{41}[\Delta,\chi_1] \Pe,
\end{split}\]
and we used $\Pe [\Delta,\chi_1]R_{21} = R_{21} [\Delta,\chi_1] \Pe= 0$, from \eqref{eq:R21U}. It remains to show that
\begin{equation}\label{e:b01pieces}
 B_{01,0} = - \frac 1 {2\pi} U_0 \otimes U_0, \qquad B_{01,-2}= 
    -\frac{\rho^2}{2}\mathcal P_e {\bf 1}_{> \rho} \Pi_{-2} {\bf 1}_{> \rho} \mathcal P_e.
\end{equation}

2. Now we turn to showing the first of \eqref{e:b01pieces}.    For $\phi \in \mch_c$, by  Lemma \ref{l:B2kp} $B_{01}\phi\in \mcg_0$,
so that by  Lemmas \ref{l:sGreens} and \ref{l:bdrypair} we find
$\langle   1, [\Delta,\chi_1]B_{01}\phi\rangle =0$.  Since $P$ is self-adjoint, $B_{00}^*=B_{00}+cB_{01}$ for some constant $c$.  We
 use   these and  the formula \eqref{eq:R2jk} for $R_{01}$ to write, for $\phi,\;\psi \in \mch_c$,
\begin{equation}\label{e:b010r1}
- 2\pi \langle B_{01,0} \phi, \psi \rangle = \langle (1-\chi_1+[\Delta,\chi_1]B_{00})\phi,1 \rangle \langle 1, (1-\chi_1+[\Delta,\chi_1]B_{00})\psi \rangle,
\end{equation}
By Lemma \ref{l:sGreens},
\[
 \langle [\Delta,\chi_1]B_{00})\phi,1 \rangle = - \langle (1-\chi_1)PB_{00}\phi,1 \rangle_{|x|<r_1} + \bpr(B_{00}\phi,1).
\]
Plugging in $PB_{00} =  I - \Pe$, and taking $r_1$ large enough that $\phi(x) = 0$ when $|x| \ge r_1$, gives
\[
 \langle (1-\chi_1+[\Delta,\chi_1]B_{00})\phi,1 \rangle =   \langle(1-\chi_1)\Pe \phi,1 \rangle + \bpr(B_{00}\phi,1) =  \bpr(B_{00}\phi,1),
\]
since $\Pe \phi$ contains no $l=0$ modes in \eqref{e:fcexp}. But, from the coefficient of $ \log \lambda$ in  Lemma \ref{l:bpl},
 \begin{equation}\label{eq:monday}
 \bp_{r_1}(B_{01}\phi,\overline{R}_{00}(\bullet, y_0))+ \bp_{r_1}(B_{00}\phi,\overline{R}_{01}(\bullet, y_0))=0,
 \end{equation}
where we used (by \eqref{eq:R21U})   $\bp_{r_1}(B_{-20}\phi,\overline{R}_{21}(\bullet, y_0))=0$. Using also the formulas \eqref{eq:R2jk} for $R_{00}$ and $R_{01}$, the fact (from Lemma \ref{l:B2kp}) that $B_{01}\phi \in \mcg_0$, and Lemma \ref{l:bdrypair} for  boundary pairing, gives 
\[
\bpr(B_{00}\phi,1) =- 2\pi \bp_{r_1}(B_{00}\phi,\overline{R}_{01}(\bullet, y_0)) = 2\pi \bpr (B_{01}\phi,\overline{R}_{00}(\bullet, y_0)) = - 2 \pi c_0(B_{01}\phi),
\]
and hence
\begin{equation}\label{e:b010c0}
 \langle B_{01,0} \phi, \psi \rangle = - 2 \pi c_0(B_{01}\phi) \bar c_0(B_{01}\psi).
\end{equation}
Since $B_{01}\phi \in \mcg_0$, \eqref{e:b010c0} shows $B_{01,0}=0$ when $\mcg_0 = \mcg_{-1}$ because in that case $c_0(B_{01}\phi)=0$ for all $\phi$. If, on the other hand,  $\mcg_0 \ne\mcg_{-1}$, then \eqref{e:b010r1} shows that there is $U_0 \in \mcg_0$ with $c_0(U_0) = 1$ such that $B_{01,0} = \beta U_0 \otimes U_0$ for some constant $\beta$. Plugging $B_{01,0}\phi = \beta\langle \phi ,U_0 \rangle U_0$, $c_0(B_{01,-2} \phi)=0$  into \eqref{e:b010c0}, and similarly for $\psi$, gives $\beta = - 2\pi \beta^2$, so it remains to show $\beta \ne 0$.

To show $\beta \ne 0$, follow the corresponding computation from our proof in the corresponding case covered in Section \ref{s:rs0}, equations \eqref{e:ba01phi} to \eqref{e:a00phi}. More specifically, we will show that $B_{01}\phi \ne 0$ for $\phi = [\Delta,\chi_1]\log|x|$. By \eqref{eq:monday}, it is enough to show that $B_{00} \phi = \log |x| + O(1)$, and this can be shown by using  \eqref{eq:B00} to show that $B_{00} \phi= O(\log |x|)$, followed by the fact that $\phi_1:=B_{00}\phi - (1-\chi_1)\log|x|$ is in the nullspace of $P$ and hence by Corollary \ref{c:g0glog} we have $\phi_1 = O(1)$. 

3. To show the second of  \eqref{e:b01pieces},  let  $\phi_j\in \mch_c$, $\psi_j = [\Delta, \chi_1] \Pe \phi_j$, and write
\[
-\langle B_{01,-2} \phi_1, \phi_2 \rangle = -\langle R_{41} \psi_1, \psi_2 \rangle = \frac {1}{32\pi} \int_{\mathbb R^2 \times \mathbb R^2} (x \cdot y)^2 \psi_1(x)\overline{\psi_2(y)}\,dx\,dy,
\]
where  we used the formula \eqref{eq:R2jk} for $R_{41}$ and the fact that the expansion of $\psi_j$ as in \eqref{e:fcexp} has no  $l\geq 1$ terms. Next, using polar coordinates and taking $a_j, b_j$, such that $\Pe\phi_j = r^{-2}(a_j \cos2\theta + b_j \sin 2 \theta) + O(r^{-3})$,
 we write this as
\[
 \begin{split}
  \int_0^{2\pi}\!\!\!\int_0^{2\pi} \!\! \cos^2(\theta_1-\theta_2&)  (a_1 \cos2\theta_1 + b_1 \sin 2 \theta_1)(\overline{a_2} \cos2\theta_2 + \overline{b_2} \sin 2 \theta_2)\, d \theta_1 \, d \theta_2 \\  & \times \frac {1}{32\pi}\Big(\int_0^\infty r^3 [\partial_r^2 + \tfrac 1 r \partial_r,\chi_1] r^{-2} dr\Big)^2 
= \frac {\pi}4 (a_1 \overline{a_2} + b_1 \overline{b_2}).
 \end{split}
\] 

Meanwhile, 
\[\begin{split}
\langle \mathcal P_e &{\bf 1}_{> \rho} \Pi_{-2} {\bf 1}_{> \rho} \mathcal P_e \phi_1, \phi_2 \rangle =  
\int_{\rho}^\infty r^{-3}dr \times \\ &\int_0^{2\pi} (a_1 \cos2\theta + b_1 \sin 2 \theta)(\overline{a_2} \cos2\theta + \overline{b_2} \sin 2 \theta)d\theta 
= \frac{\pi}{2\rho^2}  (a_1 \overline{a_2} + b_1 \overline{b_2}),
\end{split}\] 
which implies $ \langle \Pe[\Delta,\chi_1]R_{41}[\Delta,\chi_1] \Pe \phi_1,\phi_2\rangle  = \frac{{\rho}^2}{2}\langle \mathcal P_e {\bf 1}_{> \rho} \Pi_{-2} {\bf 1}_{> \rho} \mathcal P_e \phi_1, \phi_2 \rangle$, as desired.
\end{proof}

\subsection{Simplification of negative logarithmic powers}\label{s:presshift}
In this section we simplify some of the negative powers of $\log \lambda$ in the resolvent expansion. Our first result shows, similarly to the first part of Theorem~\ref{t:resexp} but without requiring $\mcg_{-2} = \{0\}$, that these negative powers vanish when $P$ has an $s$-resonance but no $p$-resonance, i.e. when $\mcg_{0}\ne\mcg_{-1}$ but $\mcg_{-1} = \mcg_{-2}$.

\begin{prop}\label{p:noneglog} If $\mcg_{0}\ne \mcg_{-1}$ and $\mcg_{-1} = \mcg_{-2}$, then   \eqref{eq:reb} can be improved to
$$R(\lambda)=  - \Pe\lambda^{-2} + \sum_{j=0}^\infty \sum_{k=0}^{k_0(j)} B_{2j,k}\lambda^{2j}(\log \lambda)^k.$$ 
\end{prop}

\begin{proof} We begin by showing that $B_{0,-1}=0$. By \eqref{e:b2j2kcomm} and Proposition \ref{p:b-2-1}, we have
\begin{equation}\label{e:b0-10}
 B_{0,-1} = B_{0,-1}[\Delta,\chi_1]R_{01}[\Delta,\chi_1]B_{0,-1}.
\end{equation}
But from the coefficient of $(\log z)^{-1}$ in \eqref{eq:lco},  
\begin{equation}\label{eq:Bjnegk}(1-\chi_1)B_{0,-1} = R_{01}[\Delta,\chi_1]B_{0,-2} +R_{00}[\Delta,\chi_1]B_{0,-1}.
\end{equation}
Hence, by the formulas \eqref{eq:R2jk} for $R_{01}$ and $R_{00}$, $B_{0,-1}\mch_c \subset \mcg_{\log}$. Since $\mcg_{0}\ne\mcg_{-1}$, by Corollary~\ref{c:g0glog} this implies $B_{0,-1}\mch_c \subset \mcg_{0}$. However, by Lemma \ref{l:Ucommint}, $R_{01} [\Delta,\chi_1] \mcg_0 = \{0\}$, and putting that into \eqref{e:b0-10} gives $B_{0,-1}=0$.

From $B_{0,-1}=0$ and \eqref{e:b2j2kcomm}, we get $B_{0,-k} =0$ for all $k \ge 2$. Finally, if there is $J \ge 0 $ such that $B_{2j,-k} = 0$ for all $j \le J$ and for all $k \ge 1$, then \eqref{e:b2j2kcomm} with $j = J+1$ shows that  $B_{2(J+1),-k}=0$. Hence, by induction, $B_{2j,-k}=0$ for all $j$ and for all $k \ge 1$.
\end{proof}

We next  simplify the leading negative powers of $\log \lambda$ in the cases  $\mcg_{0}= \mcg_{-1}$ or $\mcg_{-1} \ne \mcg_{-2}$. Since $\mcg_{-1} \ne \mcg_{-2}$ is the more complicated case, we begin with the simpler case $\mcg_0 = \mcg_{-1}=\mcg_{-2}$, i.e. the case of no resonance at zero. This result is analogous to the second part of Theorem \ref{t:resexp} but without the assumption $\mcg_{-2}= \{0\}$, just as Proposition \ref{p:noneglog} is analogous to the first part of Theorem \ref{t:resexp}.

\begin{prop} \label{p:easysum} Suppose $P$ has no resonance at zero, i.e. $\mcg_0 = \mcg_{-2}$. There exists $U_{\log} \in \mcg_{\log}$ with $c_{\log}(U_{\log}) = 1$,  such that, if $a= \gamma_0+c_0(U_{\log})$, then
$$\sum_{k=1}^\infty B_{0,-k} (\log \lambda)^{-k}= \sum_{k=1}^\infty(\log \lambda)^{-k} a^{k-1} B_{0,-1} =\frac{1}{ 2\pi(\log \lambda-a)} U_{\log} \otimes U_{\log}.$$
\end{prop}

\begin{proof} 1. We will show that there exist $U_{\log} \in \mcg_{\log}$ and constants $\beta_k$ such that if $k \ge 1$ then
\begin{equation}\label{e:b0-kckUk}
B_{0,-k} = \beta_k U_{\log} \otimes U_{\log}.
\end{equation}
To show \eqref{e:b0-kckUk}, by \eqref{e:b2j2kcomm} applied twice, first with $j=0$ and $k'=1$,  and second with $j=0$ and $k'=k$, if $k \ge 1$ then
\begin{equation}\label{eq:B0-k}
B_{0,-k}=B_{0,-1}[\Delta,\chi_1]R_{01}[\Delta,\chi_1]B_{0,-k} = B_{0,-k}[\Delta,\chi_1]R_{01}[\Delta,\chi_1]B_{0,-1}.
\end{equation}
Hence \eqref{e:b0-kckUk} for general $k$ follows from  \eqref{e:b0-kckUk} for $k=1$, and by  \eqref{e:b2jksa} and $R_{01}[\Delta,\chi_1]\mathcal{P}_e=0$ it is enough to show that $B_{0,-1}\mch_c \subset \mcg_{\log}$. To get    $B_{0,-1}\mch_c \subset \mcg_{\log}$, observe that  the coefficient of $(\log z)^{-k}$ in \eqref{eq:lco} gives
\begin{equation}\label{eq:expB0-k}
(1-\chi_1)B_{0,-k}\phi =R_{00}[\Delta,\chi_1]B_{0,-k}\phi + R_{01}[\Delta,\chi_1]B_{0,-k-1}\phi,
\end{equation}
and hence $B_{0,-1}\mch_c \subset \mcg_{\log}$ follows from the formulas \eqref{eq:R2jk} for $R_{00}$ and $R_{01}$ together with $P B_{0,-1} = 0$ from the coefficient of $(\log \lambda)^{-1}$ in $(P-\lambda^2)R(\lambda)=I$.

2. We construct $U_{\log}$ such that $c_{\log}(U_{\log}) = 1$, and show that then $\beta_1=\frac 1 {2\pi}$. By Proposition~\ref{p:b-2-1}, $B_{-2,-k}=0$ for $k \ge 1$, and by Proposition \ref{p:b01form}, $B_{01}=-\mathcal{P}_eB_{01}\mathcal{P}_e$.  
Then the coefficient of $\log z$ in \eqref{eq:lco}, and the formulas \eqref{eq:R2jk} for $R_{01}$ and $R_{00}$, show that for any $\phi \in \mch_c$, 
$$0= R_{01}(1-\chi_1+[\Delta, \chi_1]B_{00})\phi +O(|x|^{-1}).$$ Hence, if $\phi \in \mch_c$, then $\langle (1-\chi_1+[\Delta, \chi_1]B_{00})\phi, 1 \rangle=0$ and  so
\begin{equation}\label{eq:zero}
R_{00}(1-\chi_1+[\Delta, \chi_1]B_{00})\phi=O(|x|^{-1}).\end{equation}
Now let $\phi_0 = \Delta \chi_1$ and $U_{\log} = B_{0,-1}\phi_0$. From \eqref{eq:B00}, and using \eqref{eq:zero} and Lemma \ref{l:Ucommint}, we have
\[
 B_{00} \phi_0 =  R_{01} [\Delta, \chi_1] B_{0,-1}\phi_0 + O(|x|^{-1}) = c_{\log} (B_{0,-1}\phi_0)  + O(|x|^{-1}) ,
\]
so with  $\phi_1 = 1-\chi_1 - B_{00} \phi_0$, we have $\phi_1 \in \mcg_{0} = \mcg_{-2}$, and hence $c_{\log} (U_{\log}) = 1$. To show that $\beta_1=\frac 1{2\pi}$, use Lemma \ref{l:sGreens} to see that calculate
\[ U_{\log}=B_{0,-1}\phi_0 = \beta_1 \langle \phi_0,U_{\log}\rangle U_{\log} = \beta_1 \langle [\Delta,\chi_1]1,U_{\log}\rangle U_{\log}=\beta_1 \bp (1,U_{\log})U_{\log}=2\pi\beta_1U_{\log}.
\]

3. It remains to show that $\beta_{k+1} = a \beta_k$ for all $k\geq 1$.   We replace $\phi$ by $\phi_0=\Delta\chi_1$ in \eqref{eq:expB0-k}  and  
take the boundary pairing of both sides of \eqref{eq:expB0-k} with $U_{\log}$ to get, since 
 $B_{0,-k} \phi_0 = \beta_k U_{\log}=\beta_k(\phi_0)U_{\log}$,
\begin{equation}\label{eq:bpns}
\beta_k \bp(U_{\log}, U_{\log}) = \beta_k \bp(R_{00}[\Delta,\chi_1]U_{\log}, U_{\log})+ \beta_{k+1}\bp(R_{01}[\Delta,\chi_1]U_{\log}, U_{\log}).
\end{equation}
By the definition \eqref{eq:bdrypair} of $\bp$, we have $\bp(U_{\log}, U_{\log}) =0$.  By Lemmas \ref{l:bdrypair} and \ref{l:Ucommint}, we have $\bp(R_{01}[\Delta,\chi_1]U_{\log}, U_{\log})=2\pi$.  
Using Lemma \ref{l:bdrypair} and writing 
$R_{00}[\Delta,\chi_1]U_{\log}= \log |x|- \gamma_0+O(|x|^{-1})$ yields
$\bp(R_{00}[\Delta,\chi_1]U_{\log},U_{\log})= -2\pi a$.  
Using these three in \eqref{eq:bpns} shows $\beta_{k+1} = a \beta_k$, as desired.
\end{proof}

For the $p$-resonance case $\mcg_{-1} \ne \mcg_{-2}$ we need the following boundary pairing calculations.

\begin{lem}\label{l:bplp}
Let $v,\; w\in \Complex^2$ be such that there are corresponding $U_v$, $U_w\in \mcg_{-1}$ satisfying $v_{-1}(U_v)=v$, $v_{-1}(U_w)=w$.  If $\chi_1(x)=0$ for  $|x|>r_1-1$, then
\begin{equation}\label{e:r21uwuv}
 \bpr(R_{21}[\Delta, \chi_1]U_w,U_v)=\pi w\cdot v,
\end{equation} 
\begin{equation}\label{e:r20uwuv}
\bpr (R_{20}[\Delta, \chi_1]U_w,U_v)=  \pi(\log r_1- \gamma_0) w \cdot v + O(r_1^{-1}).
\end{equation}
\end{lem}

\begin{proof}
Equation \eqref{e:r21uwuv} follows from  \eqref{eq:R21U}. To prove \eqref{e:r20uwuv}, first note that from the resolvent kernel formulas \eqref{eq:R2jk}, putting $ \widetilde R_{21}(x,y) = (\log|x| - \gamma_0 - 1)R_{21}(x,y)$, we have
\[
 R_{20}(x,y) = \widetilde R_{21}(x,y) - \tfrac 1 {8\pi} x \cdot y + O(1), \qquad 
 \partial_r R_{20}(x,y) = \partial_r\Big(\widetilde R_{21}(x,y) - \tfrac 1 {8\pi} x \cdot y\Big) + O(|x|^{-1}).
\]
Next, from \eqref{e:r21uwuv} and the definition \eqref{eq:bdrypair} of $\bpr$,
\[
 \bpr( \widetilde R_{21}[\Delta,\chi_1]U_w,U_v) = \pi (\log r_1 - \gamma_0 - 1) w \cdot v  - r_1^{-1}\int_{|x|=r_1} (R_{21} [\Delta,\chi_1]U_w)\overline{U_v},
\]
while using \eqref{eq:R21U} again gives
\[
 \int_{|x|=r_1} (R_{21} [\Delta,\chi_1]U_w)\overline{U_v} = - \frac 1{2r_1^2} \int_{|x|=r_1} (w \cdot x)(x \cdot v)= - \frac {\pi}{2 }r_1 w \cdot v.
\]
By Lemma \ref{l:sGreens} and \eqref{eq:R21U} again,
\[
 \langle [\Delta, \chi_1]U_w, x \cdot \bullet \rangle = \bpr (U_w, x \cdot \bullet ) = 2 \pi w \cdot x,
\]
so that 
\[
 \bpr  (- \tfrac 1 {8\pi} \langle [\Delta, \chi_1]U_w, x \cdot \bullet \rangle, U_v) = \frac \pi 2 w \cdot v,
\]
which implies   \eqref{e:r20uwuv}.
\end{proof}

\begin{prop}\label{p:B-2negk}
If $\mcg_{-1} = \mcg_{-2}$, then $B_{-2,-k}=0$ for all $k\geq 1$.  If $\mcg_{-1} \ne \mcg_{-2}$, use the functions $U_{w_m}$ of Proposition \ref{p:b-2-1} to define
$$\alpha_m=\lim_{r_1 \to \infty} \Big(\frac{1}{\pi}\langle U_{w_m},U_{w_m}\rangle_{|x|<r_1} - \log r_1\Big).$$
Then, with  $M = \dim(\mcg_{-1} / \mcg_{-2})$, we have
$$\sum_{k=1}^\infty  (\log \lambda)^{-k} B_{-2,-k} = \frac{1}{\pi}\sum_{m=1}^{M}(\log \lambda -\gamma_0-\alpha_m)^{-1} U_{w_m}\otimes U_{w_m},$$ 
provided if $M=2$ that $U_{w_1}$ and $U_{w_2}$ are chosen such that 
\begin{equation}\label{eq:fortho}
\langle U_{w_1},U_{w_2}\rangle_{|x|<r_1} \to 0 \;\text{as $r_1 \to \infty.$}\end{equation}
\end{prop}

\begin{proof} 1. Proposition \ref{p:b-2-1} already proves the first statement of the proposition. 

Next we show that if $M=2$ it is possible to choose $U_{w_1},\; U_{w_2}$ so that \eqref{eq:fortho} holds.  
We begin with  $U_{w_1}, \; U_{w_2}$ from Proposition  \ref{p:b-2-1}, recalling $\{w_1,\; w_2\}$ is an orthonormal set.  Define a map 
$\Complex^2 \rightarrow \mcg_{-1} $ by $v\mapsto U_v$, where if $v=c_1w_1+c_2w_2$, $U_v= c_1U_{w_1}+c_2 U_{w_2}$.
Define a quadratic form on $\Complex^2$ by 
$$q(v,v')= \lim_{r\rightarrow \infty} \left( \langle U_{v},U_{v'}\rangle _{|x|<r_1}-\pi v\cdot v' \log r_1 \right).$$
That the limit exists follows from our expansions of $U_{w_1}$, $U_{w_2}$ at infinity.   There is a self-adjoint operator 
$Q$ so that $q(v,v')=(Qv)\cdot v'$ for all $v,\;v'\in \Complex^2$.  Let $\tilde{w}_1$, $\tilde{w_2}$ be a basis of eigenvectors for $Q$
so that $\tilde{w}_n\cdot \tilde{w}_m=\delta_{nm}$, and then let $\tilde{w}_1, \; \tilde{w}_2$ be the new $w_1$, $w_2$.

The remainder of the proof is similar to Step 3 of Proposition \ref{p:easysum}.

2. Given $\phi\in \mch_c$ and $k\in \Natural, $ there are constants $\beta_{m,-k}=\beta_{m,-k}(\phi)$ such  that $B_{-2,-k}\phi = \sum_{m=1}^M \beta_{m,-k}U_{w_m}$, and we shall show that $\beta_{m,-k-1} = (\gamma_0 + \alpha_m) \beta_{m,-k}$. 

First, using Lemma \ref{l:bplp} and $w_m \cdot w_n = \delta_{mn}$,
\begin{equation}\label{eq:dbp} \bp(R_{21}[\Delta, \chi_1]B_{-2,-k-1}\phi,U_{w_m})=\pi \beta_{m,-k-1}.
\end{equation}
To  compute the left hand side of \eqref{eq:dbp}, note that if 
$\phi_1\in L^2_c(\Real^2)$ and $\phi_1(x)=0$ if $|x|>r_1-1$, then
$\bp( R_{01}\phi_1,U_{w_m})=0=\bp( R_{00}\phi_1,U_{w_m})$ by Lemma \ref{l:bdrypair} and the resolvent formulas \eqref{eq:R2jk}.  Using this and the 
coefficient of $(\log z)^{-k}$ from \eqref{eq:lco} yields that $\bp( B_{0,-k}\phi, U_{w_m})$ equals 
\begin{equation}\label{eq:Friday}
\bp( R_{21}[\Delta, \chi_1]B_{-2,-k-1}\phi, U_{w_m})+\bp( R_{20}[\Delta, \chi_1]B_{-2,-k}\phi, U_{w_m}).
\end{equation}
By Lemma \ref{l:bpid},
\[
\bp( B_{0,-k}\phi, U_{w_m}) = \langle   B_{-2,-k}\phi, U_{w_m}    \rangle_{|x|<r_1}  =
\sum_{n=1}^M\beta_{n,-k}\langle  U_{w_n}, U_{w_m}    \rangle_{|x|<r_1},
\] 
and combining with the results of Lemma \ref{l:bplp} gives that $\pi \beta_{m, -k-1}$ equals
$$\pi ( \gamma_0 - \log r_1) \beta_{m,-k} +  \sum_{n=1}^M\beta_{n,-k}\langle  U_{w_n}, U_{w_m}\rangle_{|x|<r_1} + O(r_1^{-1})\to \pi(\gamma_0 + \alpha_m) \beta_{m,-k}.$$

3. Hence
\[\begin{split}
\sum_{k=1}^\infty  &(\log \lambda)^{-k} B_{-2,-k} \phi 
= \sum_{k=1}^\infty \sum_{m=1}^M (\log \lambda)^{-k} \beta_{m,-k}U_{w_m} \\ &= \sum_{k=1}^\infty \sum_{m=1}^M (\log \lambda)^{-k} (\gamma_0+\alpha_m)^{k-1}\beta_{m,-1}U_{w_m} = \sum_{m=1}^M \frac{ \beta_{m,-1}U_{w_m}}{\log \lambda - \gamma_0 - \alpha_m}.
\end{split}\]
The proof is completed by using the explicit expression for $B_{-2,-1}$ from Proposition \ref{p:b-2-1}.
\end{proof}

\subsection*{Acknowledgments} The authors are grateful to Maciej Zworksi for his encouragement and helpful conversations over the course of this project. The BIRS conference ``Mathematical aspects of the physics with non-self-adjoint operators'' provided motivation for considering resolvent expansions for non-self-adjoint operators, and Maciej Zworski pointed out the connection with \cite{ADH}.
We also gratefully
acknowledge partial support from Simons Collaboration Grants for Mathematicians.  KD was, in addition, 
 partially supported by NSF grant DMS-1708511.


\end{document}